\documentclass[11pt]{amsart}
\usepackage{amsmath, amssymb, amscd, mathrsfs, url, pinlabel,verbatim}
\usepackage[pagebackref]{hyperref}
\usepackage[margin=1in,marginparwidth=0.75in,centering,letterpaper,dvips]{geometry}
\usepackage{color,dcpic,latexsym,graphicx,epstopdf,comment}
\usepackage[all]{xy}
\usepackage[dvipsnames]{xcolor}
\usepackage{tikz,tikz-cd,pgfplots}
\usepackage{upquote,tabularx,textcomp}
\usepackage[shortlabels]{enumitem}
\usepackage{mathtools}
\usepackage[color=blue!20!white,textsize=tiny]{todonotes}

\title{L-spaces and knot traces}

\author[John A. Baldwin]{John A. Baldwin}
\address{Department of Mathematics \\ Boston College}
\email{john.baldwin@bc.edu}

\author[Steven Sivek]{Steven Sivek}
\address{Department of Mathematics \\ Imperial College London}
\email{s.sivek@imperial.ac.uk}

\makeatletter
\newtheorem*{rep@theorem}{\rep@title}
\newcommand{\newreptheorem}[2]{%
\newenvironment{rep#1}[1]{%
 \def\rep@title{#2 \ref{##1}}%
 \begin{rep@theorem}}%
 {\end{rep@theorem}}}
\makeatother

\newtheorem {theorem}{Theorem}
\newreptheorem{theorem}{Theorem}
\newtheorem {lemma}[theorem]{Lemma}
\newtheorem {proposition}[theorem]{Proposition}
\newtheorem {corollary}[theorem]{Corollary}

\newtheorem {question}[theorem]{Question}

\numberwithin{equation}{section}
\numberwithin{theorem}{section}

\theoremstyle{definition}
\newtheorem{definition}[theorem]{Definition}

\newtheorem{remark}[theorem]{Remark}
\newtheorem*{remark*}{Remark}

\newlist{pcases}{enumerate}{1}
\setlist[pcases]{
  label=\bf{Case~\arabic*:}\protect\thiscase.~,
  ref=\arabic*,
  align=left,
  labelsep=0pt,
  leftmargin=0pt,
  labelwidth=0pt,
  parsep=0pt
}
\newcommand{\case}[1][]{%
  \if\relax\detokenize{#1}\relax
    \def\thiscase{}%
  \else
    \def\thiscase{~#1}%
  \fi
  \item
}

\newcommand{\Z}{\mathbb{Z}}

\newcommand{\CP}{\mathbb{CP}}
\newcommand{\F}{\mathbb{F}}
\newcommand{\Q}{\mathbb{Q}}
\newcommand{\spc}{\operatorname{Spin}^c}
\newcommand{\spinc}{\mathfrak{s}}
\newcommand{\spint}{\mathfrak{t}}

\newcommand{\Img}{\operatorname{Im}}

\newcommand{\cF}{\mathcal{F}}

\newcommand{\cS}{\mathcal{S}}

\newcommand{\cptwo}{\overline{\CP}^2}

\DeclareMathOperator{\Span}{span}

\newcommand\hf{\mathit{HF}}

\newcommand\hfk{\mathit{HFK}}
\newcommand\cfk{\mathit{CFK}}
\newcommand\cfkinfty{\cfk^\infty}
\newcommand\hfkhat{\widehat{\hfk}}

\newcommand\symp{\mathrm{symp}}

\DeclareFontFamily{U}{mathx}{\hyphenchar\font45}
\DeclareFontShape{U}{mathx}{m}{n}{
      <5> <6> <7> <8> <9> <10>
      <10.95> <12> <14.4> <17.28> <20.74> <24.88>
      mathx10
      }{}
\DeclareSymbolFont{mathx}{U}{mathx}{m}{n}
\DeclareFontSubstitution{U}{mathx}{m}{n}
\DeclareMathAccent{\widecheck}{0}{mathx}{"71}
\newcommand{\HMto}{\widecheck{\mathit{HM}}}

\newcommand{\hfhat}{\widehat{\mathit{HF}}}

\newcommand{\pt}{\mathrm{pt}}

\newcommand{\godd}{\mathrm{odd}}
\newcommand{\geven}{\mathrm{even}}
\newcommand{\mirror}[1]{\overline{#1}}

\makeatletter
\DeclareFontFamily{OMX}{MnSymbolE}{}
\DeclareSymbolFont{MnLargeSymbols}{OMX}{MnSymbolE}{m}{n}
\SetSymbolFont{MnLargeSymbols}{bold}{OMX}{MnSymbolE}{b}{n}
\DeclareFontShape{OMX}{MnSymbolE}{m}{n}{
    <-6>  MnSymbolE5
   <6-7>  MnSymbolE6
   <7-8>  MnSymbolE7
   <8-9>  MnSymbolE8
   <9-10> MnSymbolE9
  <10-12> MnSymbolE10
  <12->   MnSymbolE12
}{}
\DeclareFontShape{OMX}{MnSymbolE}{b}{n}{
    <-6>  MnSymbolE-Bold5
   <6-7>  MnSymbolE-Bold6
   <7-8>  MnSymbolE-Bold7
   <8-9>  MnSymbolE-Bold8
   <9-10> MnSymbolE-Bold9
  <10-12> MnSymbolE-Bold10
  <12->   MnSymbolE-Bold12
}{}

\let\llangle\@undefined
\let\rrangle\@undefined
\DeclareMathDelimiter{\llangle}{\mathopen}%
                     {MnLargeSymbols}{'164}{MnLargeSymbols}{'164}
\DeclareMathDelimiter{\rrangle}{\mathclose}%
                     {MnLargeSymbols}{'171}{MnLargeSymbols}{'171}
\makeatother

\newcounter{desccount}

\newcommand{\descref}[1]{\hyperref[#1]{#1}}

\usetikzlibrary{calc,intersections}
\tikzset{every picture/.style=thick}
\tikzset{link/.style = { white, double = black, line width = 1.75pt, double distance = 1.25pt, looseness=1.75 }}
\tikzset{crossing/.style = {draw, circle, dotted, minimum size=0.5cm, inner sep=0, outer sep=0}}
\pgfplotsset{compat=1.12}

\begin{document}

\begin{abstract}
There has been a great deal of interest in understanding which knots are characterized by which of their Dehn surgeries. We study a 4-dimensional version of this question: which knots are determined by which of their traces? We prove several results that are in stark contrast with what is known about characterizing surgeries, most notably that the 0-trace detects every L-space knot.  Our proof combines tools in Heegaard Floer homology with results about surface homeomorphisms and their dynamics. We also consider nonzero traces, proving for instance that each positive torus knot is determined by its $n$-trace for any $n\leq 0$, whereas no non-positive integer is known to be a characterizing slope for any positive torus knot besides the right-handed trefoil. 
\end{abstract}

\maketitle

\section{Introduction}

Given a knot $K\subset S^3$ and an integer $n$, the \emph{$n$-trace} $X_n(K)$ is the smooth, oriented 4-manifold with boundary obtained from $B^4$ by attaching an $n$-framed 2-handle along $K\subset \partial B^4$. We say that $X_n(K)$ \emph{detects}  $K$ if its oriented diffeomorphism type determines $K$ --- that is, if \[X_n(J)\cong X_n(K) \] implies that $J=K$. In this paper, we propose and study the following question:

\begin{question}
\label{ques:traces}
For which knots $K\subset S^3$ and  integers $n$ does $X_n(K)$ detect $K$?
\end{question}

The $n$-trace  $X_n(K)$ is far from detecting $K$ in general \cite{akbulut-traces,lickorish-shake,AJOT}. This is prominently illustrated in Piccirillo's proof that the Conway knot is not slice, a key part of which involves finding a different knot with the same $0$-trace  \cite{piccirillo-conway}.

One way to prove that $X_n(K)$ detects $K$ is to show that its boundary $S^3_n(K)$, which is the result of $n$-framed Dehn surgery on $K$, does --- in other words, that $n$ is a \emph{characterizing slope} for $K$. A lot of effort has gone into  understanding  characterizing slopes for knots. Building on work by Lackenby  \cite{lackenby-characterizing} and McCoy \cite{mccoy-torus}, Sorya recently proved that for any knot $K$ all rational numbers with sufficiently large denominator are characterizing slopes, and if $K$ is composite then all non-integers are characterizing \cite{sorya}; see also \cite{sorya-wakelin}.  Deciding whether a given \emph{integer} characterizes a certain knot seems to be much harder in general, and this is especially true for the integer $n=0$. 

On this last point, Gabai proved in 1987 that $0$ is a characterizing slope for the unknot, figure-8, and trefoils in  \cite{gabai-foliations3}, and these were the only knots known to be characterized by their $0$-surgeries until we showed in 2022 that every \emph{nearly fibered} genus-1 knot is as well \cite{bs-nonfibered,bs-0-characterizing}. But it is still open, for example,  whether $0$ is a characterizing slope even for the torus knot $T_{2,5}$.

By contrast, our main result says that the $0$-trace detects every L-space knot:
\begin{theorem}
\label{thm:main}
If $K$ is an L-space knot then $X_0(K)$ detects $K$.
\end{theorem}

Here, $K\subset S^3$ is an \emph{L-space knot} if some positive Dehn surgery on $K$ is a Heegaard Floer L-space, meaning a rational homology 3-sphere $Y$ such that 
\[\dim \hfhat(Y) = |H_1(Y;\Z)|.\]
This class of knots includes all positive torus knots, as well as any other knot with a positive lens space surgery, like the $(-2,3,7)$-pretzel knot.

Theorem \ref{thm:main} follows from a combination of two results which may each be of independent interest. The first  says that any single trace detects \emph{whether} a given knot is an L-space knot:

\begin{theorem} \label{thm:lspace-detection-main}
If $K$ is an L-space knot and $X_n(J) \cong X_n(K)$, then $J$ is also an L-space knot and has the same genus as $K$. 
\end{theorem}

The $n=0$ case, which suffices for Theorem~\ref{thm:main}, can be deduced from our work in \cite{bs-fibered-sqp}, which in turn was inspired by our proof in \cite{bs-lspace} that \emph{instanton} L-space knots are fibered.  The general case builds on this but is considerably more involved; it further relies on results by Hayden--Mark--Piccirillo \cite{hayden-mark-piccirillo}, which also enable us to prove in Theorem \ref{thm:trace-determines-hfhat-restated} that any trace $X_n(K)$ with $n\geq 0$  determines the $\Z/2\Z$-graded Heegaard Floer homology of each positive rational surgery on $K$.

The second and more surprising ingredient in Theorem \ref{thm:main} is our result that if two L-space knots have the same 0-surgery then they are the same:
\begin{theorem} \label{thm:main-lspace-zero}
If $J$ and $K$ are L-space knots such that $S^3_0(J) \cong S^3_0(K)$, then $J=K$.
\end{theorem}
Theorem~\ref{thm:main-lspace-zero} relies on a deep relationship between the symplectic Floer homology of a surface diffeomorphism and the Heegaard Floer homology of its mapping torus, established via a combination of  work by Lee--Taubes \cite{lee-taubes} and Kutluhan--Lee--Taubes \cite{klt1}.

We used this relationship in joint work with Hu \cite{bhs-cinquefoil} to prove that if $K$ is a genus-2 hyperbolic L-space knot then the pseudo-Anosov representative of its monodromy has no fixed points, en route to proving  that Khovanov homology detects the cinquefoils; this was later used to prove that $T_{2,5}$ is the only  genus-2 L-space knot \cite{frw-cinquefoil}.     In \cite{ni-monodromy,ni-fixed}, Ni extended our argument to show  that the monodromy of \emph{any} nontrivial fibered knot  $K \subset S^3$ is freely isotopic to a map with at most
\begin{equation*}\label{eqn:dimminus1} \dim \hfkhat(K,g(K)-1) - 1 \end{equation*}
fixed points; see also Ghiggini--Spano \cite{ghiggini-spano}.  This is particularly useful when \[\dim \hfkhat(K,g(K)-1)=1,\] which holds, for instance, whenever $K$ is an L-space knot; more generally, we  call such  knots \emph{ffpf}, since they are fibered with fixed-point-free monodromy. Theorem \ref{thm:main-lspace-zero} is then a special case of our stronger result, Theorem~\ref{thm:zero-surgery-fpf}, which says  that no two ffpf knots have the same 0-surgery.

Our proof of Theorem \ref{thm:zero-surgery-fpf} goes very roughly as follows: suppose that $K$ is an ffpf knot with fiber surface  $\Sigma$ and monodromy $h$.  Let $\hat\Sigma$ be the closed surface obtained by capping  off $\Sigma$  with a disk, and let $\hat h$ be the \emph{closed  monodromy} obtained by extending $h$ by the identity over this  disk, so that $S^3_0(K)$ is the mapping torus of $\hat h$. The ffpf condition further implies  \cite{ni-exceptional} that $h$ is   \emph{veering}.  
We combine this with Ni's results about fixed points to argue that the canonical Nielsen--Thurston form of $\hat h$  has  a unique fixed point at the center of the capping disk. Removing the fiber over this fixed point then recovers the knot complement $S^3\setminus K$. 

Now let $J$ be another ffpf knot with closed monodromy $\hat g$ such  that \[S^3_0(J)\cong S^3_0(K).\] Since this manifold has a unique fibration, $\hat g$ is conjugate to $\hat h$. We  then use the rigidity of pseudo-Anosovs within their mapping classes \cite{flp} to show that the conjugating homeomorphism identifies the unique fixed points of the Nielsen--Thurston forms of $\hat g$ and $\hat h$. It follows that \[S^3\setminus J\cong S^3\setminus K,\]  and hence that $J=K$ by \cite{gordon-luecke-complement}.

\subsection{Nonzero traces}
We  consider detection by nonzero traces as well, and prove several results that are in stark contrast with what is known about characterizing slopes. For instance, all rational surgeries characterize the trefoils  \cite{osz-characterizing}, but if $K$ is a positive torus knot besides the right-handed trefoil, then it is  unknown whether \emph{any} integer $n\leq 0$ is a characterizing slope for $K$. On the other hand, we use Theorem \ref{thm:lspace-detection-main} to prove the following:

\begin{theorem}
\label{thm:negtracetorus}
If $K$ is a positive torus knot and $n\leq 0$, then $X_n(K)$ detects $K$.
\end{theorem}

This leaves open the question of whether every trace detects each torus knot, which by Theorem \ref{thm:negtracetorus} is equivalent to asking whether every \emph{positive} trace detects each \emph{positive} torus knot. We expect this to be true for torus knots of the form $T_{2,2g+1}$, and  prove the following  result in that direction. Together with Theorem \ref{thm:negtracetorus}, this result implies  that if knot Floer homology detects $T_{2,2g+1}$ then  so does every trace; this knot Floer result  is   known  to hold when $|g| \leq 2$ \cite{osz-genus,ghiggini,frw-cinquefoil}.

\begin{theorem} \label{thm:positive-trace-t2n-main}
Suppose for some knot $K$ and  positive integers $n$ and $g$ that $X_n(K) \cong X_n(T_{2,2g+1}).$  Then either $K = T_{2,2g+1},$ or $1 \leq n \leq 4g-1$ and $K$ is a hyperbolic knot such that \[\hfkhat(K) \cong \hfkhat(T_{2,2g+1})\] as bigraded vector spaces.
\end{theorem}

We remark that not every positive integer is a characterizing slope for each positive torus knot, not even for torus knots of the form $T_{2,2g+1}$. For example, \[S^3_{21}(T_{4,5}) \cong S^3_{21}(T_{2,11}).\] In Appendix~\ref{sec:torus-pairs}, we describe infinitely many new pairs of distinct positive torus knots $(K,J)$  such that $S^3_n(K) \cong S^3_n(J)$ for some integer $n>0$, generalizing a family discovered by Ni--Zhang \cite{ni-zhang-torus-1} which includes the example above. On the other hand, Proposition~\ref{prop:compare-torus-surgeries} guarantees in this case that $g(K) \neq g(J)$, and then Theorem~\ref{thm:lspace-detection-main} implies that these knots are distinguished by their $n$-traces.

\subsection{Questions} Our results  above and their  comparisons with results about characterizing slopes raise many natural questions; we record a few of these here. The first asks whether we can upgrade our main result, Theorem \ref{thm:main}, to a statement about characterizing slopes:

\begin{question}
Is $0$ a characterizing slope for every L-space knot?
\end{question}

An affirmative answer   would follow from Theorem \ref{thm:main-lspace-zero} if one could show that  $0$-surgery, like the $0$-trace, detects whether a given knot is an L-space knot.

The class of ffpf knots contains but is much broader than the class of all nontrivial L-space knots, leading to the question below. Note that an affirmative answer would follow from Theorem \ref{thm:zero-surgery-fpf} if one could show that the 0-trace detects whether a knot is ffpf.

\begin{question}
Does the $0$-trace detect every ffpf knot?
\end{question}

One of the most important results on characterizing slopes is Lackenby's proof \cite{lackenby-characterizing} that every knot has one. On the other hand, Baker--Motegi showed in \cite{baker-motegi} that the knot $8_6$ has no \emph{integral} characterizing slope, which inspired the following question in the original version of this paper:

\begin{question}
Is every knot detected by at least one of its traces?
\end{question}

This question has since been answered in the negative by Baker--Kegel--Motegi in \cite{baker-kegel-motegi}.

Finally, our proof of Theorem \ref{thm:main} relies crucially on the fact that L-space knots are fibered with veering, fixed-point-free monodromy, as this guarantees that the closed monodromy has a unique fixed point such that the complement of the fiber over this fixed point recovers the knot complement. It is natural to ask whether there is an analogue of the ffpf condition for \emph{non-fibered} knots which guarantees something similar, for example:

\begin{question}
\label{ques:nonfibered}
Is there a condition on the knot Floer homology of a non-fibered hyperbolic knot $K\subset S^3$ which guarantees that there is  a pseudo-Anosov flow on $S^3_0(K)$ with exactly one orbit which generates the first homology of this surgery and whose complement recovers $S^3\setminus K$?
\end{question}

In \cite{baldwin-velavick}, Baldwin--Vela-Vick proved that if $K\subset S^3$ is a nontrivial fibered knot then \[\dim \hfkhat(K,g(K)-1) \geq \dim \hfkhat(K,g(K)) = 1.\] The ffpf condition is that this inequality is an \emph{equality}. So, perhaps the condition sought in Question \ref{ques:nonfibered} for a non-fibered knot $K$ is also that \[\dim \hfkhat(K,g(K)-1) = \dim \hfkhat(K,g(K)).\]
As noted in \cite{baldwin-velavick}, it is conjectured that the first is always greater than or equal to the second for non-fibered knots as well; this is still open.

\subsection{Conventions} We will only consider Heegaard Floer homology with coefficients in $\F = \Z/2\Z$. We will use $\cong$ to indicate orientation-preserving diffeomorphism throughout.

\subsection{Organization}  In \S\ref{sec:veering}, we review and prove some facts about surface diffeomorphisms and their dynamics, which we then  apply to prove Theorem~\ref{thm:zero-surgery-fpf} and consequently Theorem \ref{thm:main-lspace-zero}. In \S\ref{sec:large-surgeries}, we prove Theorem~\ref{thm:trace-large-surgery}, which says that if $n\geq 0$  then $X_n(K)$ determines the Heegaard Floer homology of large surgeries on $K$. It follows that the $0$-trace detects whether a  knot is an L-space knot, which together with Theorem \ref{thm:main-lspace-zero} implies our main result, Theorem~\ref{thm:main}. In \S\ref{sec:rational-surgeries}, we combine Theorem~\ref{thm:trace-large-surgery} with results from \cite{hayden-mark-piccirillo} to prove in Theorem~\ref{thm:trace-determines-hfhat-restated} that any trace $X_n(K)$  determines the Heegaard Floer homology of many or all rational surgeries on $K$; this quickly implies Theorem~\ref{thm:lspace-detection-main} as well.

The remainder of the paper is devoted to applying these results to nonzero traces of torus knots. In \S\ref{sec:torus-surgeries}, we prove Theorem~\ref{thm:negtracetorus} by combining Theorem~\ref{thm:lspace-detection-main} with   results about characterizing slopes for torus knots.  In \S\ref{sec:t2n-positive}, we focus on the torus knots $T_{2,2g+1}$,   proving Theorem~\ref{thm:positive-trace-t2n-main}. Finally, we describe in  Appendix~\ref{sec:torus-pairs} new infinite families of distinct positive torus knots with diffeomorphic $n$-surgeries. As explained above, the corresponding $n$-traces are not diffeomorphic, providing another illustration of the difference between trace detection and characterizing slopes.

\subsection{Acknowledgements} We thank Matt Hedden and Tye Lidman for helpful conversations, and the referee for a very careful reading and lots of detailed feedback on the original version of this paper.

\section{Zero-surgeries on ffpf knots} \label{sec:veering}
The ultimate goal of this section is to prove Theorem \ref{thm:zero-surgery-fpf}; we will see in Lemma~\ref{lem:ffpf-examples} that this subsumes Theorem~\ref{thm:main-lspace-zero}, thereby proving the latter as well. We first establish some results about monodromies of fibered knots and their behaviors under \emph{capping off}.

\subsection{Surface homeomorphisms and fibered knots}

Let $h:\Sigma\to\Sigma$ be a homeomorphism of a compact oriented surface  with possibly empty boundary. By Thurston's classification \cite{thurston-diffeomorphisms}, $h$ is  isotopic rel boundary to a map $\varphi:\Sigma\to\Sigma$ such that
\begin{itemize}
\item there exists a possibly empty \emph{reducing system} $\Gamma \subset \Sigma$ consisting of a finite disjoint union of simple closed curves which is fixed setwise by $\varphi$;
\item if $S$ is a component of $\Sigma\setminus \Gamma$ and $n$ is the smallest positive integer such that $\varphi^n(S) = S$, then $\varphi^n|_S$ is freely isotopic to a periodic or pseudo-Anosov map. When $n=1,$  we  refer to  this map and $S$ as a periodic or pseudo-Anosov component of $h$ or $\Sigma \setminus \Gamma$, accordingly.
\end{itemize}
We will assume that $\Gamma$ is minimal with respect to inclusion, in which case it  is unique up to isotopy \cite[Theorem~C]{blm} and is called the \emph{canonical reducing system} for $h$. We  refer to the pair $(\varphi,\Gamma)$ as the \emph{Nielsen--Thurston form} of $h$.

We will focus hereafter on the Nielsen--Thurston forms of monodromies of fibered knots in $S^3$ and their associated \emph{closed monodromies}; we spend some time below establishing notation, terminology, and some preliminary results specific to that setting.

Let $K\subset S^3$ be a fibered knot with fiber surface $\Sigma$ and monodromy \[h:\Sigma\to \Sigma.\] Then  $\Sigma$ has one boundary component and $h$ restricts to the identity on $\partial\Sigma$. Following \cite{hkm-veering},   $h$ is said to be  \emph{right-veering} if for every properly embedded arc $\alpha \subset \Sigma$, either $\alpha$ and $h(\alpha)$ are isotopic rel boundary, or  $h(\alpha)$ is  to the right of $\alpha$ near each endpoint after these arcs are isotoped to intersect minimally. We say simply that $h$ is  \emph{veering} if either $h$ or $h^{-1}$ is right-veering.

Let $(\varphi, \Gamma)$ be the Nielsen--Thurston form of $h$. The suspension of $\Gamma$ in the mapping torus of $\varphi$ is a collection of incompressible tori in the knot complement $S^3 \setminus \nu(K)$. The complement of these tori consists of Seifert fibered pieces swept out by the periodic components of $\Sigma\setminus \Gamma$, and hyperbolic pieces swept out by the pseudo-Anosov components \cite{thurston-fiber}. 

Let  $\Sigma_0$ be the component of $\Sigma \setminus \Gamma$ containing $\partial \Sigma$; we call this the \emph{outermost} component of $\Sigma \setminus \Gamma$.  Note that  $\varphi$ fixes $\Sigma_0$ setwise, and its restriction to $\Sigma_0$ is  freely isotopic to a map  \[\varphi_0: \Sigma_0 \to \Sigma_0\] which is either periodic or pseudo-Anosov. The \emph{fractional Dehn twist coefficient} \cite{hkm-veering} \[c(h)\in \Q\]  records the amount of twisting near $\partial \Sigma$ in this free isotopy, giving a measure of how veering $h$ is.  (If $\varphi_0$ is pseudo-Anosov, then $1/c(h)$ is the \emph{degeneracy slope} \cite{gabai-oertel} of the suspension flow of $\varphi_0$.)  When the outermost component is pseudo-Anosov, we have the following \cite[\S3]{hkm-veering}:

\begin{theorem} \label{thm:fdtc-veering}
If $\varphi_0$ is pseudo-Anosov, then $h$ is veering if and only if $c(h)\neq 0$.
\end{theorem}

Moreover, using work of Gabai \cite{gabai-problems}, Kazez--Roberts showed  \cite[Theorem~4.5]{kazez-roberts} that these twist coefficients are highly constrained for monodromies of fibered knots in $S^3$: 

\begin{proposition}\label{prop:0-or-1/n}
If $K \subset S^3$ is a fibered knot with monodromy $h$, then  $|c(h)|<1$.
\end{proposition}

It follows from the discussion above that $K$ is hyperbolic if and only if  the outermost map $\varphi_0$ is pseudo-Anosov and $\Gamma = \emptyset$. In the periodic case, we have the following, cf.\ \cite[Theorem~2]{budney-jsj}:

\begin{proposition}\label{prop:periodic-cable}
If  $\varphi_0$ is periodic, then either
\begin{itemize}
\item $\Gamma = \emptyset$ and $K$ is a $(p,q)$-torus knot, or 
\item $\Gamma \neq \emptyset$ and either
\begin{itemize}
\item $K$ is a composite knot and $\Sigma_0$ is a planar surface, or
\item $K$ is a $(p,q)$-cable knot of some other knot in $S^3$ and $\Sigma_0$ has genus $g(T_{p,q})$.
\end{itemize}
\end{itemize}
In each case above, $q \geq 2$ and $\gcd(p,q) = 1$.
\end{proposition}

\begin{proof}
The first case is due to Seifert \cite{seifert}; the second is a special case of \cite[Proposition~4.2]{kazez-roberts}.  The statement of that proposition did not include the case where $K$ is composite, but  the proof allows for the possibility that a closed orbit in the mapping torus of $\phi_0$ is a meridian of $K$. In this case, Kazez--Roberts  show that $c(h) = 0$, which implies that $\phi_0$ restricts to the identity on $\partial \Sigma\subset \Sigma_0$. It then follows from \cite[Lemma~2.6]{bns}, which is  really an application of \cite[Lemma~1.1]{jiang-guo}, that $\phi_0$ is the identity on all of $\Sigma_0$.  But then the outermost piece of the JSJ decomposition of $S^3 \setminus \nu(K)$, given by the mapping torus of $\phi_0$, is simply $Y_0=\Sigma_0 \times S^1$.  From here, \cite{budney-jsj} says that $\Sigma_0$ must be a planar surface with at least three boundary components, and that  $K$ is a connected sum of the  nontrivial knots whose complements are the connected components of $(S^3\setminus \nu(K)) \setminus Y_0$.
\end{proof}

Let $\hat\Sigma = \Sigma\cup D$ be the closed surface obtained by capping off the fiber of $K$ with a disk, and let \[\hat h:\hat\Sigma\to\hat\Sigma\] be the map obtained by extending the monodromy $h$ by the identity over the capping disk $D$. We call $\hat h$  the \emph{closed monodromy} of $K$; note that its mapping torus  is $S^3_0(K)$. Let \[\hat\Sigma_0 = \Sigma_0 \cup D.\] Let $(\hat\varphi,\hat\Gamma)$ be the  Nielsen--Thurston form of $\hat h$. The Nielsen--Thurston forms of $h$ and  $\hat h$ can be very different in general (for instance, they can be freely isotopic to pseudo-Anosov and periodic maps, respectively), but this is ruled out under some mild hypotheses: 

\begin{theorem} \label{thm:rv-fill-monodromy}
If the monodromy $h$ is veering and $\Sigma_0$ is not a pair of pants, then
\begin{itemize}
\item $\hat\Gamma\subset \Gamma$ and $\varphi$ agrees with $\hat\varphi$ on $\Sigma\setminus \Sigma_0 = \hat\Sigma\setminus\hat\Sigma_0$;\vspace{1mm}
\item if $\varphi_0$ is periodic, then the pseudo-Anosov components of $\Sigma\setminus\Gamma$ and $\hat\Sigma\setminus\hat{\Gamma}$ are equal; \vspace{1mm}
\item if $\varphi_0$ is pseudo-Anosov, then $\hat\Gamma= \Gamma$ and the restriction of $\hat\varphi$ to $\hat\Sigma_0$ is freely isotopic to a pseudo-Anosov map  $\hat\varphi_0$ which agrees with   $\varphi_0$ on $\Sigma_0$ and fixes one point $p$ in $D$.
\end{itemize}
\end{theorem}

In brief, the hypothesis that $h$ is veering guarantees that if $\varphi_0$ is pseudo-Anosov then its stable and unstable invariant foliations have at least two prongs on $\partial \Sigma$. This implies that  $\Sigma_0$ caps off to a pseudo-Anosov component with one additional fixed point in the capping disk while the other components are unchanged. The hypothesis  that $\Sigma_0$ is not a pair of pants guarantees that if $\varphi_0$ is periodic then capping off preserves all pseudo-Anosov components. This theorem will be crucial for understanding  the fixed points of the pseudo-Anosov components of the closed monodromies of L-space knots and more generally ffpf knots in our proof of Theorem \ref{thm:zero-surgery-fpf}.

\begin{proof}[Proof of Theorem \ref{thm:rv-fill-monodromy}]
We first show that $\varphi_0$ extends to a map \[\hat\varphi_0: \hat\Sigma_0 \to \hat\Sigma_0\] of the same type (periodic or pseudo-Anosov) as $\varphi_0$. If $\varphi_0$ is periodic, then it rotates $\partial \Sigma$ by some rational multiple of $2\pi$, and we simply extend this rotation radially across the disk $D$ to define $\hat\varphi_0$. 

Now suppose that $\varphi_0$ is pseudo-Anosov.  By definition, this means that there is a pair of transverse singular measured  foliations $(\cF_s,\mu_s)$ and $(\cF_u,\mu_u)$ of $\Sigma_0$, called the \emph{stable} and \emph{unstable} foliations of $\varphi_0$, such that
\[ \varphi_0(\cF_s,\mu_s) = (\cF_s,\lambda^{-1}\mu_s) \quad\text{and}\quad \varphi_0(\cF_u,\mu_u) = (\cF_u,\lambda\mu_u), \]
for some real number $\lambda > 1$. The interior singularities of these foliations  have  at least 3 \emph{prongs}, and  each of these foliations has some number $n \geq 1$ of \emph{boundary prongs} ending on $\partial \Sigma \subset \partial \Sigma_0$. Since  $\varphi_0$ preserves the boundary prongs, the fractional Dehn twist coefficient of $h$ is given by  \[c(h) = k/n\] for some integer $k$; see \cite[\S3.2]{hkm-veering}. Since $h$ is veering, Theorem~\ref{thm:fdtc-veering} says that $k\neq 0$. Meanwhile, Proposition~\ref{prop:0-or-1/n} implies that $|k|<n$. It follows that $n\geq 2$. In this case,  $\varphi_0$ extends  to a pseudo-Anosov homeomorphism $\hat\varphi_0$ of $\hat\Sigma_0$, and $\cF_s$ and $\cF_u$ extend to stable and unstable foliations $\hat\cF_s$ and $\hat\cF_u$ for $\hat\varphi_0$ in which the $n$ boundary prongs extend to $n$ prongs meeting at point $p$ in $D$. The point $p$ is a singularity of these two foliations when $n\geq 3$, and a smooth point when $n=2$; see Figure \ref{fig:prongs}. Note  that $p$ is the unique fixed point of $\hat\varphi_0$  in $D$.

(For the reader interested in more details about how one extends the foliations and transverse measures over the disk $D$, we take the following easy way out: a pseudo-Anosov homeomorphism of a surface with boundary is, by definition, a map which restricts on the complement of the boundary to a pseudo-Anosov map of the corresponding punctured surface. Capping off a boundary component from this point of view corresponds simply to filling in the corresponding puncture with a point.)

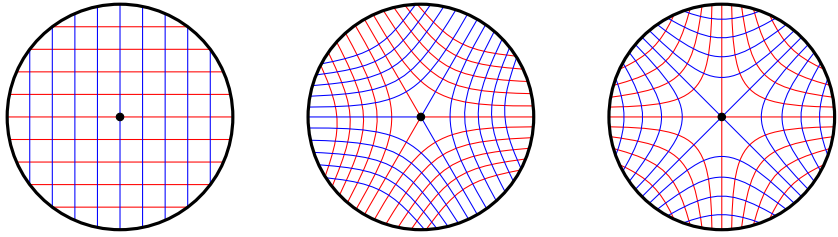
\begin{figure}[ht]
\begin{tikzpicture}[every path/.style={thin}]
\foreach \i in {2,3,4} { \coordinate (center\i) at (4*\i,0); }

\begin{scope} 
{ \tikzset{every path/.style={}}
\clip (center2) circle (1.5); }
\foreach \j in {-6,...,6} { \draw[red] (center2) ++ (-2,0.3*\j) -- ++(4,0); }
\foreach \j in {-6,...,6} { \draw[blue] (center2) ++ (0.3*\j,-2) -- ++(0,4); }
\end{scope}

\begin{scope} 
{ \tikzset{every path/.style={}}
\clip (center3) circle (1.5); }
\foreach \j in {0,120,240} {
  \draw[red] (center3) -- ++(\j:2);
  \foreach \n in {2,4,...,10} {
    \draw[red] ($(center3)+(\j:2)+(\j+90:0.075*\n)$) .. controls ($(center3)+(\j+60:0.1*\n-0.1/\n)$) .. ($(center3)+(\j+120:2)+(\j:0.075*\n)$);
  }
}
\foreach \j in {60,180,300} {
  \draw[blue] (center3) -- ++ (\j:2);
  \foreach \n in {2,4,...,10} {
    \draw[blue] ($(center3)+(\j:2)+(\j+90:0.075*\n)$) .. controls ($(center3)+(\j+60:0.1*\n-0.1/\n)$) .. ($(center3)+(\j+120:2)+(\j:0.075*\n)$);
  }
}
\end{scope}

\begin{scope} 
{ \tikzset{every path/.style={}}
\clip (center4) circle (1.5); }
\foreach \j in {0,90,180,270} {
  \draw[red] (center4) -- ++(\j:2);
  \foreach \n in {3,6,...,18} {
    \draw[red] ($(center4)+(\j:2)+(\j+90:0.03*\n+0.1/\n)$) .. controls ($(center4)+(\j+45:0.1*\n-0.3/\n)$) .. ($(center4)+(\j+90:2)+(\j:0.03*\n+0.1/\n)$);
  }
}
\foreach \j in {45,135,225,315} {
  \draw[blue] (center4) -- ++ (\j:2);
  \foreach \n in {3,6,...,18} {
    \draw[blue] ($(center4)+(\j:2)+(\j+90:0.03*\n+0.1/\n)$) .. controls ($(center4)+(\j+45:0.1*\n-0.3/\n)$) .. ($(center4)+(\j+90:2)+(\j:0.03*\n+0.1/\n)$);
  }
}
\end{scope}

\foreach \i in {2,3,4} {
\draw[very thick] (center\i) circle (1.5);
\draw[black,fill=black] (center\i) circle (0.05);
}

\end{tikzpicture}
\caption{The transverse foliations $\hat\cF_s$ and $\hat\cF_u$ on $D$, in the cases where $\cF_s$ and $\cF_u$ each have $n=2,3,4$ boundary prongs.  When $n=2$ the point $p$ at the center is a smooth point of each foliation, while for $n\geq 3$ it is an $n$-pronged singular point.} \label{fig:prongs}
\end{figure}

It follows in either case that $\Gamma$ is a reducing system for $\hat h$ as well, and thus contains $\hat\Gamma$. Moreover, $\hat\varphi$ is the extension of $\varphi$ by the identity over $D$, and in particular agrees with $\varphi$ on \[\Sigma \setminus \Sigma_0 = \hat\Sigma \setminus \hat\Sigma_0,\] proving the first item of the theorem. For the other items, we analyze the difference $\Gamma \setminus \hat\Gamma$.
 
We first claim that $\Gamma\setminus \hat\Gamma$ consists of components of $\partial \hat\Sigma_0$. Indeed, $\hat\varphi$ fixes $\Gamma$ and thus $\partial \hat\Sigma_0$ setwise. It follows that $\hat\Gamma\cup \partial \hat\Sigma_0\subset \Gamma$ is also a reducing system for $\hat h,$ and hence for $h$ as well. But then the minimality of $\Gamma\subset \Sigma$ as the canonical reducing system for $h$ implies that \[\Gamma = \hat\Gamma \cup \partial \hat\Sigma_0,\] which proves our claim about $\Gamma\setminus \hat\Gamma$.

We next claim that every component of $\Gamma$ is homotopically essential in $\hat\Sigma$, and that no two components of $\Gamma$ are freely homotopic in $\hat\Sigma$. Indeed, the analogous conditions hold for $\Gamma \subset \Sigma$ by the minimality of this reducing system, and the only way this fails upon capping off is if $\hat\Sigma_0$ is a disk or an annulus;  equivalently, if $\Sigma_0$ is an annulus or a pair of pants.  Minimality of $\Gamma\subset \Sigma$ rules out the first possibility, while the second is excluded  by the hypotheses of the theorem. 

We now use these two claims to prove the rest of the theorem.   The third item follows from what we have already established if we can show that $\hat\Gamma = \Gamma$ when $\varphi_0$ is pseudo-Anosov, or equivalently if we can show that $\hat\Gamma \neq \Gamma$ implies that $\Sigma_0$ is not a pseudo-Anosov component of $\Sigma\setminus \Gamma$. 

Let us therefore assume that  $\hat\Gamma \neq \Gamma$.  Let $\hat\Sigma_1$ be the component of $\hat\Sigma\setminus \hat\Gamma$ containing $\hat\Sigma_0$. By the first claim,  this must be a proper containment, and we can write \[\hat\Sigma_1=\hat\Sigma_0 \cup S_1\cup \dots \cup S_k,\] where  $S_1,\dots,S_m$ are some components of $\Sigma\setminus \Gamma$   adjacent to $\Sigma_0$. For the third item in the theorem, it suffices to show that $\Sigma_0$ is not a pseudo-Anosov component, as discussed; for the second item, it suffices to show that none of the $S_j$ are pseudo-Anosov components. We  prove both at once.

Each $S_j$ shares some boundary component $\gamma$ with $\hat\Sigma_0$. Some power $(\hat\varphi)^n$ restricts to a map on $\hat\Sigma_1$ which is freely isotopic either to the identity or to a pseudo-Anosov map. Since  $\hat\varphi$ fixes $\Gamma$ setwise and $\gamma$ is a component of $\Gamma$, we can also assume that $(\hat\varphi)^n$ fixes  $\gamma$. But the second claim implies that $\gamma$ is homotopically essential and not boundary parallel in $\hat\Sigma_1,$ which means that $(\hat\varphi)^n|_{\hat\Sigma_1}$ is not freely isotopic to a pseudo-Anosov map; it is therefore freely isotopic to the identity. But then $(\hat\varphi)^n$ restricts to maps on $\hat\Sigma_0$ and $S_j$ which are  both freely isotopic to the identity. 

Now, the restrictions of $(\hat\varphi)^n$ and $\varphi^n$ to $S_j$ agree since  $\varphi = \hat\varphi$ outside of $D$. It follows that $\varphi^n|_{S_j}$ is freely isotopic to the identity, and hence that $S_j$ is a periodic component of $\Sigma\setminus \Gamma$.  We can also conclude that $\Sigma_0$ is a periodic component of $\Sigma\setminus \Gamma$. Indeed, if it is not then $\varphi_0$ is   pseudo-Anosov, in which case the same is true of $\hat\varphi_0$, as we showed at the beginning of the proof. This would imply that the restriction of $(\hat\varphi)^n$ to $\hat\Sigma_0$ is freely isotopic to a pseudo-Anosov map, but we have just shown that it is freely isotopic to the identity, a contradiction.
It follows that $\Sigma_0$ and $S_j$ are both periodic components of $\Sigma\setminus \Gamma$, completing the proof of the theorem.
\end{proof}

\begin{corollary} \label{cor:not-veering}
If  $K \subset S^3$ is a fibered hyperbolic knot such that  $S^3_0(K)$ is not hyperbolic, then its monodromy is not veering.
\end{corollary}

\begin{proof}
The monodromy $h:\Sigma\to \Sigma$ of $K$ is freely isotopic to a pseudo-Anosov map $\varphi_0$. In this case, $\Sigma_0 = \Sigma$ is not a pair of pants, since it has connected boundary.  Then $h$ is not veering; otherwise, Theorem \ref{thm:rv-fill-monodromy} would imply that $\hat\varphi_0$ is also pseudo-Anosov, which would imply that $S^3_0(K)$, the mapping torus of $\hat h$, is hyperbolic.
\end{proof}

\subsection{Fixed points and ffpf knots} 

If $K \subset S^3$ is a fibered knot of genus $g \geq 1$, then Baldwin and Vela-Vick proved in \cite{baldwin-velavick}  that
\[ \dim \hfkhat(K, g-1) \geq 1. \]
We will focus on knots for which  equality is achieved, motivating the following definition:

\begin{definition} \label{def:ffpf}
A knot $K \subset S^3$ is \emph{ffpf} (for ``fibered and fixed-point-free") if it is fibered of some genus $g \geq 1$, and 
\[ \dim \hfkhat(K,g-1) = 1. \] We will justify this terminology in Proposition~\ref{prop:pA-fixed-points}, which says that if a knot is ffpf then no pseudo-Anosov component of its monodromy can have fixed points. 
\end{definition}

This class of knots includes all nontrivial L-space knots and most \emph{almost L-space knots} \cite{bs-characterizing}, where  $K\subset S^3$ is  said to be an \emph{almost L-space knot}  if \[\dim\hfhat(S^3_m(K)) = m+2\] for sufficiently large integers $m$:

\begin{lemma} \label{lem:ffpf-examples}
Any nontrivial L-space knot, or almost L-space knot of genus $g\geq 3$, is ffpf.
\end{lemma}

\begin{proof}
L-space knots are fibered \cite{ghiggini,ni-hfk}, with $\dim \hfkhat(K,g-1) \leq 1$  \cite[Theorem~1.2]{osz-lens}.  Almost L-space knots of genus $g \geq 2$ are fibered  \cite[Proposition~3.9]{bs-characterizing}, and if $g \geq 3$ then the inequality $\dim \hfkhat(K,g-1) \leq 1$ follows from \cite[Remark~3.12]{bs-characterizing}; see also \cite{binns-almost}. The lemma then follows from Baldwin--Vela-Vick's  inequality in the other direction \cite{baldwin-velavick}.
\end{proof}

\begin{lemma} \label{lem:hfk-outer-monodromy}
If  $K$ is an ffpf knot, then it is prime, it is not a $(\pm1,q)$-cable for any $q \geq 2$, and its monodromy is veering. 
\end{lemma}

\begin{proof}
Suppose $K$ has genus $g\geq 1$. If $K$ were a nontrivial connected sum, then the proof of \cite[Corollary~1.4]{baldwin-velavick}  would show that $\dim \hfkhat(K,g-1) \geq 2$.  If instead $K$ were the $(\pm1,q)$-cable of some companion knot $C$, with $q \geq 2$, then their Alexander polynomials  would satisfy
\[ \Delta_K(t) = \Delta_C(t^q), \]
with nonzero $t^i$-coefficient only when $i$ is a multiple of $q$; but then $g = q\cdot g(C)$ implies that $\Delta_K(t)$ has $t^{g-1}$-coefficient $a_{g-1} = 0$, and so
\[ \dim \hfkhat(K,g-1) \equiv \chi(\hfkhat_\ast(K,g-1)) = a_{g-1} = 0 \pmod{2}. \]
In either case, this  contradicts $\dim \hfkhat(K,g-1) = 1$, so $K$ must be prime and not a $(\pm1,q)$-cable. Lastly, the monodromy of $K$ is  veering by \cite[Theorem~A.1]{ni-exceptional}.
\end{proof}

The ffpf condition has the following geometric consequence, which in combination with Theorem \ref{thm:rv-fill-monodromy} will ultimately let us recover a ffpf knot from its 0-surgery in the next subsection: 

\begin{proposition} \label{prop:pA-fixed-points}
Let $K$ be a ffpf knot with monodromy $h: \Sigma \to \Sigma$, and let $(\varphi,\Gamma)$ be the Nielsen--Thurston form of $h$. Then for each pseudo-Anosov component $S$ of $\Sigma\setminus \Gamma$ that is fixed setwise by $\varphi$, the restriction of $\varphi$ to $S$ is freely isotopic to a pseudo-Anosov map with no fixed points.
\end{proposition}

\begin{proof}
Suppose that $K$ has genus $g\geq 1$. In the proof of  \cite[Theorem~1.2]{ni-fixed},  Ni constructs a closed surface $P = \Sigma\cup \Sigma'$ and a homeomorphism $\sigma: P \to P$ such that
\begin{itemize}
\item $\sigma|_{\Sigma} = h$, which implies in particular that $\sigma$ fixes $\partial \Sigma$ pointwise;\vspace{1mm}
\item $\sigma|_{\Sigma'}$ is the monodromy of some fibered cable knot $L' \subset S^1\times S^2$;\vspace{1mm}
\item the symplectic Floer homology of the mapping class of $\sigma$ is given by
\[ \dim \hf^\symp(P,\sigma) = \dim \hfkhat(K,g-1) - 1; \]
\item if  $(\varphi',\Gamma')$ is the Nielsen--Thurston form of  $\sigma|_{\Sigma'}$, then \[(\varphi\cup \varphi',\, \Gamma \cup \partial \Sigma \cup \Gamma')\] is the Nielsen--Thurston form of $\sigma$. 
\end{itemize}
Cotton-Clay describes  how to compute the symplectic Floer homology of a mapping class in \cite[Theorem~4.2]{cotton-clay}; see also \cite[Theorem~1.3]{ni-fixed}. Most importantly for us, interior fixed points of pseudo-Anosov components  contribute linearly independent elements to Floer homology. In particular, if $S$ is a pseudo-Anosov component of $\Sigma \setminus \Gamma$ and hence of $P \setminus (\Gamma\cup\partial \Sigma \cup \Gamma')$, with $\varphi(S) = S$, and if \[\varphi|_S = (\varphi\cup \varphi')|_S\] is freely isotopic to a pseudo-Anosov map with $m$ fixed points, then 
 \[m\leq \dim \hf^\symp(P,\sigma).\] The ffpf condition on $K$, combined with the third item above, then says that \[m\leq  \dim \hf^\symp(P,\sigma) = \dim \hfkhat(K,g-1) - 1 = 0. \] Therefore, $m=0$ as claimed in the proposition.
\end{proof}

\subsection{Zero-surgeries on ffpf knots} Here, we apply our results  from the previous two subsections to prove the main result of this section, Theorem \ref{thm:zero-surgery-fpf}. By Lemma \ref{lem:ffpf-examples},  this theorem immediately implies Theorem \ref{thm:main-lspace-zero}, which says that any two L-space knots are distinguished by their $0$-surgeries.

\begin{theorem} \label{thm:zero-surgery-fpf}
If  $J$ and $ K$ are  ffpf knots such that $S^3_0(J) \cong S^3_0(K)$, then $J=K$.
\end{theorem}

We begin by proving some special cases.

\begin{lemma} \label{lem:torus-knots-different-0-surgery}
If $J$ and $K$ are torus knots such that $S^3_0(J) \cong S^3_0(K)$, then $J=K$.
\end{lemma}

\begin{proof}
Write $J = T_{a,b}$ and $K = T_{c,d}$, and let $A,B,C,D$ denote the absolute values of $a,b,c,d$.  The Alexander polynomial of $S^3_0(J)$ is equal to the Alexander polynomial of $J$, and likewise for $K$, so
\[ \Delta_J(t) = \Delta_{S^3_0(J)}(t) = \Delta_{S^3_0(K)}(t) = \Delta_K(t). \]
Observe that
\[ \Delta_J(t) = t^{-\frac{1}{2}(A-1)(B-1)} \frac{(t^{AB}-1)(t-1)}{(t^{A}-1)(t^{B}-1)} \]
recovers both $AB$ (as the maximal order of one of its roots) and
\[ A+B = AB + 1 - (A-1)(B-1) = AB+1 - \deg \Delta_{J}(t), \]
and therefore uniquely determines the set $\{A,B\}$ as the roots of the polynomial \[x^2-(A+B)x+AB.\]  The same is true for $K$, so $\{A,B\} = \{C,D\}$, and thus $J$ is isotopic to either $K$ or its mirror.

Now, Casson and Gordon \cite[Lemma~3.1]{casson-gordon} proved that $S^3_0(K)$ also determines the signature of $K$, as minus the Casson--Gordon invariant $\sigma_1(S^3_0(K),\chi)$ associated to the unique surjection \[\chi: H_1(S^3_0(K)) \twoheadrightarrow \Z/2\Z,\] so $\sigma(J) = \sigma(K)$.  Note that $\sigma(J)$ is negative if $J$ is a positive torus knot (see for example \cite{rudolph-signature}) and positive otherwise, and likewise for $\sigma(K)$, so $J$ and $K$ are either both positive torus knots or both negative torus knots.  We must therefore have $J=K$ rather than $J=\mirror{K}$, as claimed.
\end{proof}

We will  prove an analogue of Lemma~\ref{lem:torus-knots-different-0-surgery} for cable knots after establishing  the following general fact about satellite knots, cf.\ \cite[Lemma~3.3]{gordon}:

\begin{lemma} \label{lem:pattern-0-homology}
Let $P \subset S^1\times D^2$ be a knot with winding number $w \geq 0$, and let $V_0$ be the result of performing $0$-surgery on $P$ in the solid torus. Then \[H_1(V_0) \cong \Z \oplus (\Z/w\Z).\]  In particular, if $w \neq 1$ then $V_0$ does not embed in $S^3$.
\end{lemma}

\begin{proof}
Let $E_P = (S^1\times D^2) \setminus \nu(P)$ be the exterior of $P$ in the solid torus.  Define peripheral curves $\mu_C,\lambda_C \subset S^1 \times \partial D^2$ in $\partial E_P$ by
\[\mu_C = \{\pt\} \times \partial D^2, \quad \lambda_C = S^1 \times \{\pt\}.\]
Let $\mu_P \subset \partial \nu(P)$ be a meridian of $P$, so that $[\mu_C] = w[\mu_P]$ in $H_1(E_P)$, and let $\lambda_P \subset \partial \nu(P)$ be the peripheral curve dual to $\mu_P$ such that $[\lambda_P] = w[\lambda_C]$ in $H_1(E_P)$. 

The homology $H_1(E_P)$ is freely generated by $[\mu_P]$ and $[\lambda_C]$. Since $V_0$ is the result of Dehn filling $E_P$ along $\lambda_P$, it follows that $H_1(V_0)$ is the quotient of $H_1(E_P)$ by $[\lambda_P] = w[\lambda_C]$, which is $\Z \oplus (\Z/w\Z)$ with the summands generated by the images of $[\mu_P]$ and $[\lambda_C],$ as claimed.

Now suppose that $w \neq 1$, so that $H_1(V_0) \not\cong \Z$.  Since the boundary of $V_0$ is a torus, any embedding $V_0 \subset S^3$ would identify it as the complement of a knot: the image of the torus $\partial V_0$ must bound a solid torus in $S^3$, say with core circle $\gamma$, so either $V_0$ is the solid torus $\overline{\nu(\gamma)}$ (that is, an unknot complement) or it is the complement $S^3 \setminus \nu(\gamma)$.  But this would mean that $H_1(V_0) \cong \Z$ after all, so no such embedding can exist, completing the proof of the lemma.
\end{proof}

\begin{proposition} \label{prop:cables-different-0-surgery}
Let $J$ and $K$ be cable knots, neither of which is a $(\pm1,q)$-cable for any $q\geq 2$, and suppose that $S^3_0(J) \cong S^3_0(K)$.  Then $J=K$.
\end{proposition}

\begin{proof}
We may assume that $J$ is the $(p,q)$-cable of some nontrivial companion knot $C_J$, where $q \geq 2$ is the winding number of the cable and $\gcd(p,q) = 1$.  Let $V_{p,q}$ be the result of $0$-surgery on the cable pattern in the solid torus; this is Seifert fibered over $D^2(q,|pq|)$ with incompressible boundary \cite[Lemma~7.2]{gordon}.

We claim that $V_{p,q}$ forms a piece of the JSJ decomposition of
\[ S^3_0(J) \cong \big( S^3 \setminus \nu(C_J) \big) \cup V_{p,q}. \]
The only way it could fail to do so is if the outermost JSJ piece $W$ of $S^3 \setminus \nu(C_J)$ were also Seifert fibered, with the same fiber slope on $\partial \nu(C_J)$ as the fiber slope on $\partial V_{p,q}$, so that $W \cup V_{p,q}$ is a single Seifert fibered component for $S^3_0(J)$.  But the Seifert fibration on $V_{p,q}$ extends the fibration on the complement of the cable pattern, so in this case $W$ and the complement of the cable pattern would also have formed a single Seifert fibered component of the exterior $S^3 \setminus \nu(J)$, and this is impossible for cable knots in $S^3$.

Lemma~\ref{lem:pattern-0-homology} says that $V_{p,q}$ does not embed in $S^3$, but since $S^3 \setminus \nu(J)$ certainly does, we see that
\begin{itemize}
\item $V_{p,q}$ is the unique JSJ piece of $S^3_0(J)$ that does not embed in $S^3$; and\vspace{1mm}
\item we can uniquely recover $S^3 \setminus \nu(C_J)$ from $S^3_0(J)$ as the complement of this piece.
\end{itemize}
We may similarly assume that $K$ is the $(r,s)$-cable of $C_K$, and let $V_{r,s}$ be the result of $0$-surgery on the corresponding cable pattern; then the same applies to $V_{r,s}$ as a JSJ piece of $S^3_0(K)$.  A homeomorphism $S^3_0(J) \cong S^3_0(K)$ identifies their respective JSJ decompositions, and therefore restricts to a homeomorphism
\[ V_{p,q} \cong V_{r,s} \]
between the unique pieces of either decomposition that do not embed in $S^3$. It similarly restricts to a homeomorphism
\[ S^3 \setminus \nu(C_J) \cong S^3 \setminus \nu(C_K) \]
between the complements of these pieces.  The latter of these tells us that $C_J = C_K$, by \cite{gordon-luecke-complement}.

Since $V_{p,q}$ and $V_{r,s}$ are homeomorphic, and they are Seifert fibered in a unique way over $D^2(q,|pq|)$ and $D^2(s,|rs|),$ respectively, comparing the orders of the singular fibers shows that $q=s$ and $|p|=|r|$.  In other words, if $J$ is the $(p,q)$-cable of $C=C_J$, then we have shown that  $K$ must be the $(\pm p,q)$-cable of the same $C$.  Just as in the proof of Lemma~\ref{lem:torus-knots-different-0-surgery}, we know that if $S^3_0(J) \cong S^3_0(K)$ then $J$ and $K$ have the same signature.  Shinohara \cite[Theorem~9]{shinohara} gave a general formula for the signatures of satellite knots, which tells us that if $J$ and $K$ are the $(p,q)$- and $(-p,q)$-cables of $C$ then we have
\begin{align*}
\sigma(J) &= \begin{cases} \sigma(T_{p,q}), & q \text{ even} \\ \sigma(T_{p,q}) + \sigma(C), & q \text{ odd}, \end{cases} &
\sigma(K) &= \begin{cases} \sigma(T_{-p,q}), & q \text{ even} \\ \sigma(T_{-p,q}) + \sigma(C), & q \text{ odd}. \end{cases}
\end{align*}
In either case, we deduce that $\sigma(T_{p,q}) = \sigma(T_{-p,q})$, and again by \cite{rudolph-signature} this is impossible unless $T_{p,q}$ is unknotted, i.e., $p=\pm 1$.  But we have excluded this last possibility by hypothesis, so $J$ and $K$ must both be the $(p,q)$-cable of $C$ and therefore $J=K$ as claimed.
\end{proof}

We are now ready to prove Theorem~\ref{thm:zero-surgery-fpf}.

\begin{proof}[Proof of Theorem~\ref{thm:zero-surgery-fpf}]
Suppose that $J$ and $ K$ are ffpf knots such that $S^3_0(J) \cong S^3_0(K)$. Since these knots have diffeomorphic $0$-surgeries, they must have the same genus $g \geq 1$ \cite{gabai-foliations3}. Let $\Sigma$ denote their common genus-$g$ fiber surface, and let 
\[ h_J, h_K: \Sigma \to \Sigma \]
be their monodromies.  We will  adopt all of the notation leading up to and used in Theorem \ref{thm:rv-fill-monodromy}. In particular, we let $\hat \Sigma$ be the result of capping off $\Sigma$ with a disk $D$, and denote by 
\[ \hat{h}_J, \hat{h}_K: \hat\Sigma \to \hat\Sigma \] the closed monodromies of $J$ and $K$, whose respective mapping tori are $S^3_0(J)$ and $S^3_0(K)$. Let \[(\varphi^J,\Gamma_J)\, \textrm{ and }\, (\varphi^K,\Gamma_K)\] be the Nielsen--Thurston forms of $h_J$ and $h_K$, and let \[(\hat\varphi^J,\hat\Gamma_J)\, \textrm{ and }\, (\hat\varphi^K,\hat\Gamma_K)\] be the Nielsen--Thurston forms of the closed monodromies. The restrictions of $\varphi^J$ and $\varphi^K$ to the outermost components are freely isotopic to maps 
\[\varphi_0^J:\Sigma_0^J\to \Sigma_0^J \,\,\, \textrm{ and } \,\,\, \varphi_0^K:\Sigma_0^K\to \Sigma_0^K\] which are each either periodic or pseudo-Anosov.

Since $J$ is ffpf, Lemma \ref{lem:hfk-outer-monodromy} says that $J$ is prime, and is not the $(\pm 1,q)$-cable of a nontrivial knot for any $q \geq 2$. It follows that $\Sigma_0^J$ is not a pair of pants. Indeed, if it were then $\varphi_0^J$ would be periodic (the pair of pants does not admit pseudo-Anosov homeomorphisms) with $\Gamma_J \neq \emptyset$. Since $J$ is prime, Proposition \ref{prop:periodic-cable} would then imply that $J$ is a $(p,q)$-cable for some $q\geq 2$, and that $\Sigma_0$ has genus $g(T_{p,q})$. But we would then have that   $g(T_{p,q})$ = 0, and hence that $p = \pm 1$, a contradiction. 

Lemma \ref{lem:hfk-outer-monodromy} further says that the monodromy $h_J$ is veering. We have thus shown that $\varphi_0^J$ and $\Sigma_0^J$ satisfy the hypotheses of Theorem \ref{thm:rv-fill-monodromy}, and the same is true of $\varphi_0^K$ and $\Sigma_0^K$. 

Proposition~\ref{prop:pA-fixed-points} says that no pseudo-Anosov component of $(\varphi_J,\Gamma_J)$ has any fixed points.  When we cap off to form $(\hat\varphi_J,\hat\Gamma_J)$, Theorem~\ref{thm:rv-fill-monodromy} says that the pseudo-Anosov components are unchanged, except when $\varphi^J_0$ is pseudo-Anosov, in which case $\varphi_0^J$ extends to a pseudo-Anosov component  \[\hat\varphi^J_0: \hat\Sigma^J_0 \to \hat\Sigma^J_0\] with a single fixed point $p_J$ in the capping disk $D$ while the other pseudo-Anosov components are unchanged.  Putting these results together, we have that:
\begin{itemize}
\item If no pseudo-Anosov component of $(\hat\varphi^J,\hat\Gamma_J)$ has a fixed point, then $\varphi^J_0$ is periodic.\vspace{1mm}
\item If some pseudo-Anosov component of $(\hat\varphi^J,\hat\Gamma_J)$ has a fixed point $p_J$, then $p_J \in D$ is the unique such fixed point, this component is $\hat\Sigma^J_0$, and the map $\varphi^J_0$ is pseudo-Anosov.
\end{itemize}
The same reasoning applies to $K$ and its associated monodromy data.

Now since $S^3_0(J) \cong S^3_0(K)$ has $b_1 = 1$, this manifold has a unique fibration over $S^1$. It follows that the closed monodromies $\hat{h}_J$ and $\hat{h}_K$ are conjugate up to isotopy: i.e. there is a homeomorphism
\[ \hat{f}: \hat\Sigma \to \hat\Sigma \]
such that $\hat{h}_K$ is isotopic to $\hat{f} \circ \hat{h}_J \circ \hat{f}^{-1}$.   Then their Nielsen--Thurston forms are conjugate by $\hat f$ as well, so either both have a pseudo-Anosov component with a fixed point, or neither one does.

Suppose first that neither $(\hat\varphi^J,\hat\Gamma_J)$ nor $(\hat\varphi^K,\hat\Gamma_K)$ has a pseudo-Anosov component with a fixed point.  Then we have seen that $\varphi^J_0$ and $\varphi^K_0$ are both periodic, so $J$ and $K$ are both either torus knots or cabled, by Proposition \ref{prop:periodic-cable}.  The manifold $S^3_0(J)$ is small Seifert fibered if $J$ is a torus knot \cite{moser}, and it is toroidal if $J$ is a nontrivial cable \cite[Corollary~7.3]{gordon}.  The same is true of $S^3_0(K)$ and $K$, so since $S^3_0(J) \cong S^3_0(K)$ we conclude that $J$ and $K$ are either both torus knots, or both cable knots.  In the first case, Lemma~\ref{lem:torus-knots-different-0-surgery}  tells us that $J=K$; in the second case, Proposition~\ref{prop:cables-different-0-surgery} gives the same conclusion since we know by Lemma \ref{lem:hfk-outer-monodromy} that neither is a $(\pm 1,q)$-cable with $q\geq 2$.

In the remaining case, we know from our discussion above that $(\hat\varphi^J,\hat\Gamma_J)$ and $(\hat\varphi^K,\hat\Gamma_K)$ each have exactly one pseudo-Anosov component containing a fixed point, given by   \[\hat\varphi^J_0:\hat\Sigma_0^J\to\hat\Sigma_0^J\,\,\, \textrm{ and }\,\,\,  \hat\varphi^K_0: \hat\Sigma_0^K\to \hat\Sigma_0^K,\] and furthermore that the unique fixed points  $p_J$ and $p_K$ of these maps lie in the capping disk $D$.   Since $\hat{f}$ preserves the Nielsen--Thurston forms, we conclude from the uniqueness of pseudo-Anosov maps within a mapping class \cite[Expos\'e~12]{flp}, that up to isotopy $\hat{f}$ restricts to a map from $\hat\Sigma^J_0 \to \hat\Sigma^K_0$ which intertwines  $\hat\varphi^J_0$ and $\hat\varphi^K_0$ and in particular sends $p_J$ to $p_K$. 
Since $p_J$ is in the capping disk, it is fixed by $\hat h_J$, which means that this closed monodromy restricts to a homeomorphism of $\hat\Sigma\setminus p_J,$
whose mapping torus is the knot complement $S^3\setminus J$. The same reasoning applies to $\hat h_K$. It follows that  $\hat f$ extends to a homeomorphism between these mapping tori, and therefore between the  complements of $J$ and $K$. We then conclude by \cite{gordon-luecke-complement} that $J=K$ after all.
\end{proof}

\section{Traces and large surgeries} \label{sec:large-surgeries}

Our goal in this section is to show that for any $n \geq 0$, the oriented diffeomorphism type of the $n$-trace $X_n(K)$ determines the $\Z/2\Z$-graded Heegaard Floer homology $\hfhat(S^3_m(K))$ of large integer surgeries on $K$, where $m \gg 0$ (in fact, $m \geq 2g(K)-1$ will suffice).  Specifically, the fact that $\chi(\hfhat(Y)) = |H_1(Y)|$ tells us that for integers $m > 0$ we have
\begin{align*}
\dim \hfhat(S^3_m(K)) &= \dim \hfhat_\geven(S^3_m(K)) + \dim \hfhat_\godd(S^3_m(K)) \\
m = |H_1(S^3_m(K))| &= \dim \hfhat_\geven(S^3_m(K)) - \dim \hfhat_\godd(S^3_m(K)),
\end{align*}
and so
\begin{equation} \label{eq:hfhat-odd-from-hfhat}
\dim \hfhat_\godd(S^3_m(K)) = \frac{ \dim \hfhat(S^3_m(K)) - m }{2}.
\end{equation}
When $m$ is sufficiently large, we have a surgery exact triangle
\[ \dots \to \hfhat(S^3) \xrightarrow{F_{X^\circ_m(K)}} \hfhat(S^3_m(K)) \to \hfhat(S^3_{m+1}(K)) \to \dots \]
in which $F_{X^\circ_m(K)} = 0$ because $X^\circ_m$ violates an adjunction inequality, so then
\[ \dim \hfhat(S^3_{m+1}(K)) = \dim \hfhat(S^3_m(K)) + 1 \]
and the right side of \eqref{eq:hfhat-odd-from-hfhat} is independent of $m \gg 0$.  We will show in Theorem~\ref{thm:trace-large-surgery} that the left side of \eqref{eq:hfhat-odd-from-hfhat} is completely determined by \[\hfhat(\partial X_n(K)) \cong \hfhat(S^3_n(K)),\] together with the various $\spc$ components of the cobordism map
\[ F_{X^\circ_n(K)}: \hfhat(S^3) \to \hfhat(S^3_n(K)). \]
It  will follow as a special case that the $0$-trace detects whether a given knot is an L-space knot.   Combined with Theorem \ref{thm:main-lspace-zero}, this will prove our main result, Theorem \ref{thm:main}, which states that the $0$-trace detects every L-space knot.

\subsection{Topological preliminaries} \label{ssec:topology-setup}

We begin by introducing some cobordisms and other notation which will be needed in the proof.

Given an oriented knot $K \subset S^3$ and an integer $n$, the homology $H_2(X_n(K)) \cong \Z$ is generated by a class $[\hat\Sigma_n]$, where $\hat\Sigma_n$ is constructed by gluing the core of the $n$-framed 2-handle to a Seifert surface for $K$ to produce an oriented surface of self-intersection $n$.  We can remove a smooth ball from the interior of $X_n(K)$ to get a cobordism $X^\circ_n(K): S^3 \to S^3_n(K)$.  We will often omit $K$ from the notation and simply write $X_n$ and $X^\circ_n$ for these compact 4-manifolds.

Suppose that $K$ has genus $g \geq 1$.  Following \cite{bs-fibered-sqp}, we let
\[ \spint_{n,i} \in \spc(X_n) \]
denote the unique $\spc$ structure on $X_n$ satisfying $\langle c_1(\spint_{n,i}), [\hat\Sigma_n] \rangle + n = 2i$, and recall that for $n \geq 0$, the adjunction inequality says that the cobordism map
\[ F_{X^\circ_n,\spint_{n,i}}: \hfhat(S^3) \to \hfhat(S^3_n(K)) \]
can only be nonzero when $|2i-n| + n \leq 2g-2$, or equivalently $1-g+n \leq i \leq g-1$.  In particular, if we write ${\bf 1}$ for the generator of $\hfhat(S^3) \cong \F$, then the elements
\begin{equation} \label{eq:zni}
z_{n,i} = F_{X^\circ_n,\spint_{n,i}}({\bf 1}) \in \hfhat(S^3_n(K))
\end{equation}
can only be nonzero for $1-g+n \leq i \leq g-1$.  We write
\begin{equation} \label{eq:span-z}
\cS_n = \Span\{ z_{n,i} \mid i \in \Z \} \subset \hfhat(S^3_n(K))
\end{equation}
for all $n \geq 0$.

Now for each $n \geq 0$ there is a surgery exact triangle
\begin{equation} \label{eq:surgery-triangle}
\dots \to \hfhat(S^3) \xrightarrow{F_{X^\circ_{n\vphantom{1}}}} \hfhat(S^3_n(K)) \xrightarrow{F_{W^{\vphantom{\circ}}_{n+1}}} \hfhat(S^3_{n+1}(K)) \to \dots,
\end{equation}
where we build $X^\circ_n$ by attaching an $n$-framed $2$-handle along $K$, and then $W_{n+1}$ is built by attaching a $(-1)$-framed $2$-handle along the image (in $S^3_n(K)$) of a meridian of $K$.  We observe by exactness that the cobordism map
\[ F_{W_{n+1}}: \hfhat(S^3_n(K)) \to \hfhat(S^3_{n+1}(K)) \]
is injective whenever $n \geq 2g-1$, since the map $F_{X^\circ_n}$ is then zero: indeed, the surface $\hat\Sigma_n \subset X^\circ_n$ has self-intersection $n > 2g-2$, violating the adjunction inequality.

\subsection{The span of the relative invariants of the $n$-trace}

The following key lemma is the Heegaard Floer analogue of \cite[Proposition~7.1]{bs-lspace}; the proof is essentially the same, but appears simpler since (unlike in framed instanton homology) we can work over $\F=\Z/2\Z$.

\begin{lemma} \label{lem:fw-z}
We have $F_{W_{n+1}}(z_{n,i}) = z_{n+1,i} + z_{n+1,i+1}$ for all $n \geq 0$ and all $i$.
\end{lemma}

\begin{proof}
Arguing as in \cite[Proposition~7.1]{bs-lspace}, we use the fact that the union of 2-handle cobordisms
\[ S^3 \xrightarrow{X^\circ_n} S^3_n(K) \xrightarrow{W_{n+1}} S^3_{n+1}(K) \]
is diffeomorphic to a blown-up $(n+1)$-trace cobordism
\[ S^3 \xrightarrow{X^\circ_{n+1} \# \cptwo} S^3_{n+1}(K), \]
with the diffeomorphism realized by a single handleslide.  If $E \subset X^\circ_{n+1} \# \cptwo$ is the exceptional sphere of the blow-up, then we can orient $E$ so that this diffeomorphism identifies
\[ [\hat\Sigma_n] = [\hat\Sigma_{n+1}] - [E] \]
as elements of $H_2$ of either cobordism.  We then know from \cite[Lemma~7.2]{bs-lspace} that the second homology $H_2(W_{n+1}) \cong \Z$ is generated by a surface $F_{n+1}$ of even self-intersection $-n(n+1)$, such that
\[ [F_{n+1}] = [\hat\Sigma_{n+1}] - (n+1)[E] \]
as elements of $H_2(X^\circ_{n+1} \# \cptwo)$.

Letting $\spint'_{n+1,j} \in \spc(W_{n+1})$ be the unique $\spc$ structure satisfying
\[ \langle c_1(\spint'_{n+1,j}), [F_{n+1}] \rangle = 2j, \]
we use the composition law for cobordism maps to write
\begin{gather} \label{eq:compute-F-z-step1}
\begin{aligned}
F_{W\vphantom{'}_{n+1\vphantom{,j}},\spint'_{n+1,j}}(z_{n,i}) &= F_{W_{n+1},\spint'_{n+1,j}} \circ F_{X^\circ_n,\spint_{n,i}}({\bf 1}) \\
&= \sum_{\substack{\spint \in \spc(X^\circ_n \cup W_{n+1}) \\ \spint|_{X^\circ_n} = \spint_{n,i} \\ \spint|_{W_{n+1}} = \spint'_{n+1,j}}} F_{X^\circ_n \cup W_{n+1}, \spint}( {\bf 1} ) \\
&= \sum_{\substack{\spint \in \spc(X^\circ_{n+1} \# \cptwo) \\ \langle c_1(\spint), [\hat\Sigma_n] \rangle = 2i-n \\ \langle c_1(\spint), [F_{n+1}] \rangle = 2j }} F_{X^\circ_{n+1} \# \cptwo, \spint}( {\bf 1} ).
\end{aligned}
\end{gather}
Each $\spint \in \spc(X^\circ_{n+1} \# \cptwo)$ restricts to $X^\circ_{n+1} \setminus B^4$ as $\spint_{n+1,k}$ for some $k$, and satisfies $\langle c_1(\spint), [E]\rangle = 2\ell - 1$ for some $\ell \in \Z$; the integers $k$ and $\ell$ determine $\spint$ uniquely, so we will write $\spint = \spint_{n+1,k;\ell}$.  Then
\begin{align*}
\langle c_1(\spint_{n+1,k;\ell}), [\hat\Sigma_n] \rangle &= \langle c_1(\spint_{n+1,k;\ell}), [\hat\Sigma_{n+1}] \rangle - \langle c_1(\spint_{n+1,k;\ell}), [E] \rangle  \\
&= (2k - (n+1)) - (2\ell - 1) \\
&= 2k - 2\ell - n,
\end{align*}
while
\begin{align*}
\langle c_1(\spint_{n+1,k;\ell}), [F_{n+1}] \rangle &= \langle c_1(\spint_{n+1,k;\ell}), [\hat\Sigma_{n+1}] \rangle - \langle c_1(\spint_{n+1,k;\ell}), (n+1)[E] \rangle \\
&= (2k - (n+1)) - (n+1)(2\ell-1) \\
&= 2k - 2\ell(n+1).
\end{align*}
Thus if $\spint = \spint_{n+1,k;\ell}$ satisfies $\langle c_1(\spint), [\hat\Sigma_n] \rangle = 2i-n$ and $\langle c_1(\spint), [F_{n+1}] \rangle = 2j$, then $i=k-\ell$ and $j = k-\ell(n+1)$.  This allows us to rewrite \eqref{eq:compute-F-z-step1} as
\begin{equation} \label{eq:Fw-z-big-sum}
F_{W\vphantom{'}_{n+1\vphantom{,j}},\spint'_{n+1,j}}(z_{n,i}) = \sum_{\substack{k,\ell \\ i = k-\ell \\ j = k-\ell(n+1) }} F_{X^\circ_{n+1} \# \cptwo, \spint_{n+1,k;\ell}}( {\bf 1} ).
\end{equation}
The blow-up formula for cobordism maps \cite[Theorem~3.7]{osz-triangles} now tells us that
\[ F_{X^\circ_{n+1}\#\cptwo, \spint_{n+1,k;\ell}} = \begin{cases} F_{X^\circ_{n+1},\spint_{n+1,k}}, & 2\ell-1=\pm1 \\ 0, & \text{otherwise}. \end{cases} \]
Thus in \eqref{eq:Fw-z-big-sum} all terms where $\ell \not\in \{0,1\}$ are zero; when $\ell=0$ the only possibly nonzero terms are those where $i=j=k$, and when $\ell=1$ the nonzero terms satisfy $i=k-1$ and $j=k-(n+1) = i - n$.  We sum \eqref{eq:Fw-z-big-sum} over all $j\in\Z$, breaking the right side into the terms with $\ell=0$ and the terms with $\ell=1$ respectively, to get
\[ F_{W_{n+1}}(z_{n,i}) = F_{X^\circ_{n+1} \# \cptwo, \spint_{n+1,i;0}}( {\bf 1} ) + F_{X^\circ_{n+1} \# \cptwo, \spint_{n+1,i+1;1}}( {\bf 1} ). \]
By the blow-up formula this simplifies to
\begin{align*}
F_{W_{n+1}}(z_{n,i}) &= F_{X^\circ_{n+1}, \spint_{n+1,i}}( {\bf 1} ) + F_{X^\circ_{n+1}, \spint_{n+1,i+1}}( {\bf 1} ) \\
&= z_{n+1,i} + z_{n+1,i+1},
\end{align*}
as claimed.
\end{proof}

\begin{proposition} \label{prop:fw-sn}
The subspaces $\cS_n = \Span\{z_{n,i} \mid i\in\Z\} \subset \hfhat(S^3_n(K))$ of \eqref{eq:span-z} satisfy
\[ \ker\left(F_{W_{n+1}}: \hfhat(S^3_n(K)) \to \hfhat(S^3_{n+1}(K)) \right) \subset \cS_n \]
and
\[ F_{W_{n+1}}(\cS_n) = \cS_{n+1} \]
for each $n \geq 0$, while each $\cS_{n+1}$ has preimage
\[ (F_{W_{n+1}})^{-1}(\cS_{n+1}) = \cS_n. \]
Moreover, if $K$ has genus $g \geq 1$ then $\cS_n = \{ 0 \}$ for all $n \geq 2g-1$.
\end{proposition}

\begin{proof}
The claim about $\ker(F_{W_{n+1}})$ follows from the exact triangle \eqref{eq:surgery-triangle}, which says that
\[ \ker(F_{W_{n+1}}) = \Img(F_{X^\circ_n}) \]
is spanned by the element
\[ F_{X^\circ_n}({\bf 1}) = \sum_{i \in \Z} F_{X^\circ_n,\spint_{n,i}}({\bf 1}) = \sum_{i\in\Z} z_{n,i}, \]
which evidently belongs to $\cS_n$.

To determine the image of $\cS_n$, we use the relation $F_{W_{n+1}}(z_{n,i}) = z_{n+1,i} + z_{n+1,i+1}$ of Lemma~\ref{lem:fw-z}.  This immediately implies the inclusion
\[ F_{W_{n+1}}(\cS_n) \subset \cS_{n+1}, \]
since $\cS_n$ is spanned by the elements $z_{n,i}$ and $F_{W_{n+1}}(z_{n,i}) \in \cS_{n+1}$.  For the reverse inclusion, it suffices to show that each $z_{n+1,i}$ belongs to $F_{W_{n+1}}(\cS_n)$.  We fix $j < \min(n-g,i)$ so that $z_{n,j} = 0$, whence
\begin{align*}
F_{W_{n+1}}(z_{n,j}+z_{n,j+1}+\dots+z_{n,i-1}) &= \sum_{k=j}^{i-1} (z_{n+1,k}+z_{n+1,k+1}) \\
&= z_{n+1,j} + z_{n+1,i} \\
&= z_{n+1,i}
\end{align*}
and so $z_{n+1,i}$ belongs to the image of $\cS_n$ after all.

In order to show that $(F_{W_{n+1}})^{-1}(\cS_{n+1}) = \cS_n$, we take $x \in \hfhat(S^3_n(K))$ and suppose that $F_{W_{n+1}}(x) \in \cS_{n+1}$.  Since $F_{W_{n+1}}$ sends $\cS_n$ surjectively onto $\cS_{n+1}$, we can find some $z\in\cS_n$ such that
\[ F_{W_{n+1}}(z) = F_{W_{n+1}}(x). \]
But then $x-z$ belongs to $\ker(F_{W_{n+1}}) \subset \cS_n$, hence $x = (x-z)+z$ belongs to $\cS_n$ as well.

Finally, if $n \geq 2g-1 \geq 1$ then we wish to show that $\cS_n=\{0\}$, so it suffices to show that $z_{n,i} = 0$ for all $i$.  But this is an immediate consequence of the adjunction inequality for the cobordism $X^\circ_n$, which contains the surface $\hat\Sigma_n$ of positive self-intersection: if $F_{X^\circ_n,\spint} \neq 0$ then $\spint$ must satisfy
\[ \left| \langle c_1(\spint), [\hat\Sigma_n] \rangle \right| + \underbrace{[\hat\Sigma_n] \cdot [\hat\Sigma_n]}_{=n} \leq \underbrace{2g(\hat\Sigma_n)-2}_{=2g-2} \]
and this is impossible.
\end{proof}

For any two nonnegative integers $n \leq m$, we can consider the composite cobordism
\begin{equation} \label{eq:vmn}
V_{n,m}: S^3_n(K) \xrightarrow{W_{n+1}} S^3_{n+1}(K) \xrightarrow{W_{n+2}} \dots \xrightarrow{W_m} S^3_m(K),
\end{equation}
which we interpret as a product cobordism if $m=n$.  The following is now a generalization of \cite[Proposition~10]{bs-fibered-sqp}, which corresponds to the case $n=0$.

\begin{proposition} \label{prop:ker-vmn}
If $K$ has genus $g \geq 1$ and we take $m \geq \max(n,2g-1)$, then 
\[ F_{V_{n,m}}: \hfhat(S^3_n(K)) \to \hfhat(S^3_m(K)) \]
has kernel equal to the subspace $\cS_n = \Span\{ z_{n,i} \mid i \in \Z \}$.
\end{proposition}

\begin{proof}
On the one hand we have
\begin{align*}
F_{V_{n,m}}(\cS_n) &= \left(F_{W_m}\circ \dots \circ F_{W_{n+2}} \circ F_{W_{n+1}}\right)(\cS_n) \\
&= \left(F_{W_m}\circ \dots \circ F_{W_{n+2}}\right)(\cS_{n+1}) \\
&= \dots \\
&= \cS_m = 0,
\end{align*}
by repeated application of Proposition~\ref{prop:fw-sn} and the fact that $m \geq 2g-1$.  This shows that $\cS_n$ lies in the kernel of $F_{V_{n,m}}$.

On the other hand, if $F_{V_{n,m}}(x) = 0$ then in particular we have $F_{V_{n,m}}(x) \in \cS_m$, so again we apply Proposition~\ref{prop:fw-sn} repeatedly to say for each $k=m-1,m-2,\dots,n$ that
\[ F_{W_{k+1}} \left( F_{V_{n,k}}(x) \right) = F_{V_{n,k+1}}(x) \in \cS_{k+1} \]
and hence $F_{V_{n,k}}(x) \in \cS_k$.  By the time we get to $k=n$, recalling that $V_{n,n}$ is a product cobordism and so $F_{V_{n,n}}$ is the identity, we have shown that $x = F_{V_{n,n}}(x)$ belongs to $\cS_n$.  Therefore $\ker F_{V_{n,m}}$ is a subset of $\cS_n$, as desired.
\end{proof}

\subsection{The mod 2 grading}

Heegaard Floer homology has  a canonical $\Z/2\Z$ grading, with respect to which $\hfhat(Y)$ has Euler characteristic $|H_1(Y;\Z)|$ if this order is finite, and $0$ otherwise.  We will write the graded pieces as $\hfhat_\geven(Y)$ and $\hfhat_\godd(Y)$.

\begin{lemma} \label{lem:cobordism-degree}
For all $n \geq 0$, we have $\deg(F_{X^\circ_n}) = 1$ and $\deg(F_{W_{n+1}}) = 0$ with respect to the absolute $\Z/2\Z$ gradings on $\hfhat(S^3_n(K))$ and $\hfhat(S^3_{n+1}(K))$.  Therefore, each of the elements $z_{n,i}$ from \eqref{eq:zni} has odd grading.
\end{lemma}

\begin{proof}
A $\spc$ cobordism $(Z,\spint): (Y_0,\spint|_{Y_0}) \to (Y_1,\spint|_{Y_1})$ between connected, nonempty 3-manifolds induces a cobordism map on $\hfhat$ that shifts the absolute $\Z/2\Z$ grading by
\[ \deg(F_{Z,\spint}) \equiv \frac{1}{2}(\chi(Z)+\sigma(Z) + b_1(Y_1)-b_1(Y_0)) \pmod{2}. \]
This is discussed in detail for monopole Floer homology in \cite{km-book}: see \cite[\S22.4]{km-book} for the reduction of the canonical plane field grading on $\HMto(Y)$ to a canonical $\Z/2\Z$ grading, and \cite[\S25.4]{km-book} for the computation of the mod 2 degrees of cobordism maps.  In \cite[\S42.3]{km-book} this is applied to the surgery exact triangle \eqref{eq:surgery-triangle}, and the authors show that for $n\geq 0$ the map $F_{X^\circ_n}$ has odd degree while $F_{W_{n+1}}$ and the connecting homomorphism both have even degree.

In Heegaard Floer homology, Huang--Ramos \cite{huang-ramos} lifted the relative $\Z/d(c_1(\spinc))\Z$ grading on $\hf^\circ(Y,\spinc)$ (itself a lift of the absolute $\Z/2\Z$ grading) to an analogous plane field grading, and showed that this lift is compatible with the degree formula for cobordism maps, so the same formula used for $\deg(F_{Z,\spint}) \pmod{2}$ in $\HMto$ also applies to $\hfhat$.  In particular, since $\hfhat(S^3) \cong \F$ is supported in even grading, it follows that the elements $z_{n,i} = F_{X^\circ_n,\spint_{n,i}}({\bf 1})$ have odd grading.
\end{proof}

Lemma~\ref{lem:cobordism-degree} tells us that the subspaces of each $\hfhat(S^3_n(K))$ spanned by the relative invariants $z_{n,i}$ are also entirely in odd grading, i.e.,
\[ \cS_n \subset \hfhat_\godd(S^3_n(K)) \]
for all $n \geq 0$.  We can also use it to deduce the following:

\begin{proposition} \label{prop:large-odd-span}
Suppose that $K$ has genus $g \geq 1$, and fix nonnegative integers $n \leq m$ with $m \geq 2g-1$.  Then $\cS_n = \Span \{z_{n,i} \mid i\in\Z\}$ fits into a short exact sequence
\[ 0 \to \cS_n \to \hfhat_\godd(S^3_n(K)) \xrightarrow{F_{V_{n,m}}} \hfhat_\godd(S^3_m(K)) \to 0, \]
where $V_{n,m}$ is the cobordism $S^3_n(K) \to S^3_m(K)$ of \eqref{eq:vmn}.
\end{proposition}

\begin{proof}
For any $k \geq 0$, the exact triangle \eqref{eq:surgery-triangle} has the form
\[ \dots \to \hfhat(S^3) \xrightarrow{F_{X^\circ_k}} \hfhat(S^3_k(K)) \xrightarrow{F_{W_{k+1}}} \hfhat(S^3_{k+1}(K)) \xrightarrow{\delta} \dots, \]
and we recall from \cite[\S42.3]{km-book} (as noted in the proof of Lemma~\ref{lem:cobordism-degree}) that $\delta$ has even degree.  Since its codomain $\hfhat(S^3) \cong \F$ is supported in even grading, it follows that
\[ \delta\left(\hfhat_\godd(S^3_{k+1}(K))\right) = 0 \]
and hence by exactness that $\hfhat_\godd(S^3_{k+1}(K))$ lies in the image of $F_{W_{k+1}}$.  But then $F_{W_{k+1}}$ also has even degree, so $F_{W_{k+1}}$ must restrict to a surjection
\[ F_{W_{k+1}}: \hfhat_\godd(S^3_k(K)) \to \hfhat_\godd(S^3_{k+1}(K)) \]
for all $k \geq 0$.

To conclude, we observe that $F_{V_{n,m}}$ is equal to the composition
\[ F_{W_m} \circ F_{W_{m-1}} \circ \dots \circ F_{W_{n+1}} \]
of maps which are all surjective on $\hfhat_\godd$, and hence $F_{V_{n,m}}$ is surjective on $\hfhat_\godd$ as well.  Since $m \geq 2g-1$, we use Proposition~\ref{prop:ker-vmn} to say that the kernel of $F_{V_{n,m}}$ is equal to $\cS_n$.
\end{proof}

The main result of this section is now an immediate corollary:

\begin{theorem} \label{thm:trace-large-surgery}
Fix a nontrivial knot $K \subset S^3$ and an integer $n\geq 0$.  Then for all integers $m \gg 0$, the dimension of $\hfhat_\godd(S^3_m(K))$ is equal to
\[ \dim \hfhat_\godd(\partial X_n(K)) - \dim \Span \{ F_{X^\circ_n,\spint}({\bf 1}) \mid \spint \in \spc(X^\circ_n) \}. \]
In particular, $K$ is an L-space knot if and only if the elements $\{ F_{X^\circ_n,\spint}({\bf 1}) \}$ span all of $\hfhat_\godd(S^3_n(K))$.
\end{theorem}

\begin{proof}
The short exact sequence of Proposition~\ref{prop:large-odd-span} tells us that for all $m \geq \max(2g-1,n)$, the dimension of $\hfhat_\godd(S^3_m(K))$ is equal to
\[ \dim \hfhat_\godd(S^3_n(K)) - \dim \Span\{z_{n,i} \mid i\in\Z\}. \]
But this is exactly the quantity we wanted to prove equal to $\dim \hfhat_\godd(S^3_m(K))$.
\end{proof}

\begin{remark}
The proof of Theorem~\ref{thm:trace-large-surgery} applies equally well to show that $X_n(K)$ determines the odd portion of the $\Z/2\Z$-graded framed instanton homology $I^\#(S^3_m(K))$ for $m \gg 0$; in particular, the case where $I^\#_\godd(S^3_m(K)) = 0$ says that for any fixed $n\geq 0$, the $n$-trace $X_n(K)$ determines whether or not $K$ is an instanton L-space knot.  Indeed, the key Lemma~\ref{lem:fw-z} that we used to start the proof can be replaced by the original \cite[Proposition~7.1]{bs-lspace}, and then after some mild care with coefficients the rest of the argument follows in the same way.
\end{remark}

\begin{remark}
The case $n=0$ of Theorem~\ref{thm:trace-large-surgery} was used in \cite{bs-fibered-sqp} to prove that L-space knots are fibered and strongly quasipositive, as follows: we know (at least for $g\geq2$) that $\hfhat(S^3_0(K),\spinc_{g-1})$ is nonzero \cite{osz-genus}, and its Euler characteristic is zero, so the portions in even and odd gradings are both nonzero.  Theorem~\ref{thm:trace-large-surgery} says that if $K$ is an L-space knot then $\hfhat_\godd(S^3_0(K),\spinc_{g-1}) \neq 0$ must be spanned by $z_{0,g-1}=F_{X^\circ_0,\spint_{0,g-1}}({\bf 1})$; but then $\hfhat_\godd(S^3_0(K),\spinc_{g-1}) \cong \F$, and fiberedness follows.  Then $F_{X^\circ_0,\spint_{0,g-1}}({\bf 1}) \neq 0$ ultimately implies (via the Heegaard Floer contact invariant) that the open book for $S^3$ with binding $K$ supports the tight contact structure on $S^3$, and so $K$ must be strongly quasipositive.
\end{remark}

\subsection{Zero-traces of L-space knots}

We conclude this section by proving our main result, Theorem~\ref{thm:main}, which asserts that every L-space knot is detected by its $0$-trace. 

\begin{proof}[Proof of Theorem~\ref{thm:main}]
Suppose that $K$ is an L-space knot, and that $X_0(J) \cong X_0(K)$.  Then since Theorem~\ref{thm:trace-large-surgery} says that the diffeomorphism type of the 0-trace determines whether or not a given knot is an L-space knot, we see that $J$ must be an L-space knot as well.  Now both $J$ and $K$ are ffpf by Lemma~\ref{lem:ffpf-examples}, and comparing the boundaries of the $0$-traces tells us that $S^3_0(J) \cong S^3_0(K)$, so we conclude by Theorem~\ref{thm:zero-surgery-fpf} that $J=K$.
\end{proof}

The following can also be proved by exactly the same argument.

\begin{theorem} \label{thm:almost-l-space-zero}
If $K$ is an almost L-space knot of genus at least $3$, then $X_0(K)$ detects $K$.
\end{theorem}

\section{Traces and rational surgeries} \label{sec:rational-surgeries}
The goal of this section is to prove in Theorem~\ref{thm:trace-determines-hfhat-restated} that any trace $X_n(K)$  determines the Heegaard Floer homology of many or all rational surgeries on $K$; this quickly implies Theorem~\ref{thm:lspace-detection-main} as well. For this, we combine the results of the previous section with work of Hayden--Mark--Piccirillo, who proved in \cite[Theorem~1.4]{hayden-mark-piccirillo} that for any $n \geq 0$, the $n$-trace of a knot $K\subset S^3$ detects the Heegaard Floer $\nu$ invariant of $K$ in the following sense:

\begin{theorem} \label{thm:hmp-nu}
Suppose for two knots $J,K \subset S^3$ and some integer $n \in \Z$ that  $X_n(J) \cong X_n(K)$.  Then $\nu(J) = \nu(K)$, except possibly if $n<0$ and $\{\nu(J),\nu(K)\} = \{0,1\}$.
\end{theorem}

The recipe for extracting $\nu(K)$ from $X_n(K)$ is not very concise, but it is summarized on \cite[p.~17]{hayden-mark-piccirillo}.

Ozsv\'ath--Szab\'o \cite[Proposition~9.6]{osz-rational} proved that the $\nu$-invariants $\nu(K)$ and $\nu(\mirror{K})$ of $K$ and its mirror govern the ranks of $\hfhat$ for nonzero rational surgeries on $K$.  In \cite[Proposition~10.1]{bs-concordance} and \cite[Lemma~10.4]{bs-concordance} we reinterpreted this as follows:

\begin{proposition}\label{prop:hfhat-surgeries}
Define an integer $\hat\nu(K)$ by the formula
\begin{equation} \label{eq:hat-nu}
\hat\nu(K) = \begin{cases} \hphantom{-}\max(2\nu(K)-1,0), & \nu(K) \geq \nu(\mirror{K}) \\ {-}\max(2\nu(\mirror{K})-1,0), & \nu(K) \leq \nu(\mirror{K}). \end{cases}
\end{equation}
Then there is an integer $\hat{r}_0(K)$ such that
\[ \dim \hfhat(S^3_{p/q}(K)) = q\cdot \hat{r}_0(K) + |p-q\hat\nu(K)| \]
for all nonzero, relatively prime integers $p \neq 0$ and $q > 0$.
\end{proposition}

In fact, for nonzero integers $m \geq \hat{\nu}(K)$ the formula of Proposition~\ref{prop:hfhat-surgeries} says that
\[ \dim \hfhat(S^3_m(K)) = m + (\hat{r}_0(K) - \hat{\nu}(K)) \]
and so \eqref{eq:hfhat-odd-from-hfhat} becomes
\begin{equation} \label{eq:r-nu-hfodd}
\hat{r}_0(K) - \hat{\nu}(K) = 2\dim \hfhat_\godd(S^3_m(K)), \qquad m \gg 0.
\end{equation}
We also observe that the definition of $\hat\nu(K)$ readily implies that $\hat\nu(\mirror{K}) = -\hat\nu(K)$.

\begin{remark}
By comparing \cite[Proposition~10.1]{bs-concordance} to \cite[Proposition~9.6]{osz-rational} when $p/q$ is a large positive integer, we see that when $\nu(K) \geq \nu(\mirror{K})$, the difference $\hat{r}_0(K) - \hat{\nu}(K)$ can be extracted from $\cfkinfty(K)$ by the formula
\[ \hat{r}_0(K) - \hat{\nu}(K) = \sum_{s\in\Z} (\dim H_*(\hat{A}_s)-1), \]
where the complexes $\hat{A}_s = C\{\max(i,j-s)=0\}$ are certain subquotients of $\cfkinfty(K)$ that determine $\hfhat$ of large surgeries on $K$.
\end{remark}

\begin{lemma} \label{lem:sum-nu-mirror}
For any knot $K \subset S^3$, we have $\nu(K)+\nu(\mirror{K}) \in \{0,1\}$.  Thus if $\nu(K) \geq 1$ then $\hat\nu(K) = 2\nu(K)-1$, while $\nu(K) \leq 0$ implies that $\hat\nu(K) \leq 0$.
\end{lemma}

\begin{proof}
According to \cite[Equation~(34)]{osz-rational}, we have $\nu(K) \in \{\tau(K),\tau(K)+1\}$ and similarly for $\nu(\mirror{K})$, so
\begin{align*}
\tau(K) &\leq \nu(K) \leq \tau(K)+1 \\
\tau(\mirror{K}) & \leq \nu(\mirror{K}) \leq \tau(\mirror{K})+1.
\end{align*}
Since $\tau(\mirror{K}) = -\tau(K)$, we add these inequalities to get $0 \leq \nu(K)+\nu(\mirror{K}) \leq 2$.  In fact, the sum cannot be $2$, because this would imply that $\nu(K)=\tau(K)+1$ and $\nu(\mirror{K})=\tau(\mirror{K})+1$, and these cannot both be true by \cite[\S3]{hom-smooth-concordance}, so the sum must be either $0$ or $1$.

These inequalities immediately imply that if $\nu(K) \geq 1$ then
\[ \nu(\mirror{K}) < 2-\nu(K) \leq 1 \leq \nu(K), \]
in which case Proposition~\ref{prop:hfhat-surgeries} says that $\hat\nu(K) = \max(2\nu(K)-1,0) = 2\nu(K)-1$.  Similarly, if $\nu(K) \leq 0$ then
\[ \nu(\mirror{K}) \geq -\nu(K) \geq 0 \geq \nu(K), \]
hence $\hat\nu(K) = -\max(2\nu(\mirror{K})-1,0) \leq 0$ as claimed.
\end{proof}

\begin{lemma} \label{lem:rational-dim-nu-hfodd}
For any knot $K \subset S^3$, the pair of integers
\[ \big(\nu(K), \dim \hfhat_\godd(S^3_m(K))\big) \]
for some sufficiently large $m$ determine the values of $\dim \hfhat(S^3_r(K))$ for all positive rational $r$.
\end{lemma}

\begin{proof}
Suppose first that $\nu(K) \leq 0$, and write $r=p/q$ for some relatively prime $p,q > 0$.  Then Lemma~\ref{lem:sum-nu-mirror} says that $\hat\nu(K) \leq 0$, hence Proposition~\ref{prop:hfhat-surgeries} gives us
\begin{align*}
\dim \hfhat(S^3_r(K)) &= p + q(\hat{r}_0(K) - \hat{\nu}(K)) \\
&= p + 2q\cdot \dim \hfhat_\godd(S^3_m(K)),
\end{align*}
where the second equation comes from \eqref{eq:r-nu-hfodd}.

If instead $\nu(K) \geq 1$, then Lemma~\ref{lem:sum-nu-mirror} says that $\hat{\nu}(K) = 2\nu(K)-1$.  Once we know $\hat\nu(K)$ and $\dim \hfhat_\godd(S^3_m(K))$ for $m \gg 0$, we use \eqref{eq:r-nu-hfodd} to recover $\hat{r}_0(K)$, and then Proposition~\ref{prop:hfhat-surgeries} determines $\hfhat(S^3_r(K))$ for all nonzero $r$.
\end{proof}

We can now prove that $X_n(K)$ determines the dimension of $\hfhat(S^3_r(K))$ for all rational slopes $r$ of a given sign, and often simply for all $r\in\Q$.

\begin{theorem} \label{thm:trace-determines-hfhat-restated}
Fix an integer $n \in \Z$ and a knot $K \subset S^3$.  Given a nonzero rational number $r\in\Q$, the dimension
\[ \dim \hfhat(S^3_r(K)) \]
is completely determined by the oriented diffeomorphism type of $X_n(K)$, meaning that if $X_n(J) \cong X_n(K)$ then $\dim \hfhat(S^3_r(J)) = \dim \hfhat(S^3_r(K))$, if any one of the following conditions hold:
\begin{enumerate}
\item $n \geq 0$ and $r > 0$. \label{i:n-r-positive}
\item $n \geq 0$ and $\nu(K) \geq 1$. \label{i:n-nu-positive}
\item $n \in \Z$ and $|\nu(K)| \geq 2$. \label{i:n-nu-large}
\end{enumerate}
\end{theorem}

\begin{proof}
For case~\eqref{i:n-r-positive}, Lemma~\ref{lem:rational-dim-nu-hfodd} says it is enough to show that $X_n(K)$ determines both $\nu(K)$ and $\dim \hfhat_\godd(S^3_m(K))$ for sufficiently large $m \gg 0$.  The assumption that $n \geq 0$ guarantees that $X_n(K)$ detects each of these quantities, by Theorem~\ref{thm:hmp-nu} and Theorem~\ref{thm:trace-large-surgery}, so in these cases the proof is complete.

For cases \eqref{i:n-nu-positive} and \eqref{i:n-nu-large}, we claim that the $n$-trace $X_n(K)$ detects $\hat\nu(K)$ as long as either $\nu(K) = 1$ and $n\geq 0$, or $|\nu(K)| \geq 2$ and $n$ is arbitrary.  If $\nu(K) \geq 1$ then Theorem~\ref{thm:hmp-nu} says that $X_n(K)$ detects $\nu(K)$ as long as either $n\geq 0$ or $\nu(K) \geq 2$, and in this case we have $\hat\nu(K) = 2\nu(K)-1$ by Lemma~\ref{lem:sum-nu-mirror}, so $X_n(K)$ detects $\hat\nu(K)$.  If instead $\nu(K) \leq -2$ then Lemma~\ref{lem:sum-nu-mirror} says that 
\[ \nu(\mirror{K}) -2 \geq \nu(K)+\nu(\mirror{K}) \geq 0, \]
so $\nu(\mirror{K}) \geq 2$, and then the previous argument says that $X_{-n}(\mirror{K})$ detects $\nu(\mirror{K})$ and hence $\hat\nu(\mirror{K})=2\nu(\mirror{K})-1$ for arbitrary $n$; since $\hat\nu(K) = -\hat\nu(\mirror{K})$ and $X_n(K) \cong -X_{-n}(\mirror{K})$, it follows that if $\nu(K) \leq -2$ then $X_n(K)$ detects $\hat\nu(K)$ for all $n\in\Z$ as well.

Now if $n \geq 0$ then $X_n(K)$ detects $\dim \hfhat(S^3_m(K))$ for $m\gg0$ by Theorem~\ref{thm:trace-large-surgery}, and we have also seen that if we also have $\nu(K) \not\in \{-1,0\}$ then $X_n(K)$ also detects $\hat\nu(K)$.  According to \eqref{eq:r-nu-hfodd} these two integers are enough to recover $\hat{r}_0(K)$, and then by Proposition~\ref{prop:hfhat-surgeries} this is enough to determine $\dim \hfhat(S^3_r(K))$ for all nonzero $r\in\Q$.  This establishes case~\eqref{i:n-nu-positive}, as well as the $n\geq 0$ portion of case~\eqref{i:n-nu-large}.

Supposing now that $n<0$ but $|\nu(K)| \geq 2$, we know once again from above that $X_n(K)$ detects $\hat\nu(K)$ and hence $\hat\nu(\mirror{K}) = -\hat\nu(K)$.  Since $-n$ is positive it follows from Theorem~\ref{thm:trace-large-surgery} that $X_{-n}(\mirror{K}) \cong -X_n(K)$ detects $\dim \hfhat(S^3_m(\mirror{K}))$ for $m\gg0$, hence so does $X_n(K)$, and again by \eqref{eq:r-nu-hfodd} we conclude that $X_n(K)$ determines $\hat{r}_0(\mirror{K})$.  Now Proposition~\ref{prop:hfhat-surgeries} tells us that $X_n(K)$ determines
\begin{equation} \label{eq:reverse-r}
\dim \hfhat(S^3_r(K)) = \dim \hfhat(-S^3_r(K)) = \dim \hfhat(S^3_{-r}(\mirror{K}))
\end{equation}
for all rational $r\neq 0$, because the right side is determined by $\hat\nu(\mirror{K})$ and $\hat{r}_0(\mirror{K})$, and so this completes the remaining part of case~\eqref{i:n-nu-large}.
\end{proof}

\begin{corollary} \label{cor:0-trace-all-surgeries}
Given a knot $K \subset S^3$ and \emph{any} nonzero $r\in\Q$, the dimension $\dim\hfhat(S^3_r(K))$ is completely determined by the oriented diffeomorphism type of the $0$-trace $X_0(K)$.
\end{corollary}

\begin{proof}
If $r > 0$ then we apply case~\eqref{i:n-r-positive} of Theorem~\ref{thm:trace-determines-hfhat-restated} to $X_0(K)$.  On the other hand, if $r < 0$ then by \eqref{eq:reverse-r} it is enough to compute $\dim \hfhat(S^3_{-r}(\mirror{K}))$, and since $-r > 0$ we can recover this by applying case~\eqref{i:n-r-positive} of Theorem~\ref{thm:trace-determines-hfhat-restated} to $X_0(\mirror{K}) \cong -X_0(K)$.  In either case we see that the oriented diffeomorphism type of $X_0(K)$ determines $\dim\hfhat(S^3_r(K))$, as desired.
\end{proof}

We now prove Theorem~\ref{thm:lspace-detection-main}, which asserts that $X_n(K)$ detects whether $K$ is an L-space knot of some fixed genus.

\begin{proof}[Proof of Theorem~\ref{thm:lspace-detection-main}]
Suppose that $K \subset S^3$ is an L-space knot, and that $X_n(J) \cong X_n(K)$ for some knot $J$ and some $n \in \Z$.  If $K$ is a right-handed trefoil, then at the boundaries of these traces we have $S^3_n(J) \cong S^3_n(K)$, hence $J$ is also a right-handed trefoil by \cite{osz-characterizing} and we are done.  Since there are no other L-space knots of genus 1 \cite{ghiggini}, we assume from now on that $g(K) \geq 2$.

We now claim that $\nu(K) = g(K)$.  To see this, we recall that $\tau(K) \leq \nu(K) \leq g(K)$, where the first inequality is again \cite[Equation~(34)]{osz-rational} and the second follows from the definition of $\nu(K)$.  When $K$ is an L-space knot we have $\tau(K)=g(K)$ by \cite[Corollary~1.6]{osz-lens}, so equality must hold throughout.

Since $\nu(K) = g(K) \geq 2$, we can apply case~\eqref{i:n-nu-large} of Theorem~\ref{thm:trace-determines-hfhat-restated} to conclude that $\dim \hfhat(S^3_r(J)) = \dim \hfhat(S^3_r(K))$ for all rational $r \neq 0$.   But $S^3_r(K)$ is an L-space if and only if $r \geq 2g(K)-1$ \cite{osz-rational}, so $S^3_r(J)$ is also an L-space if and only if $r \geq 2g(K)-1$, and we conclude that $J$ is an L-space knot and that $g(J)=g(K)$.
\end{proof}

Recall from \S\ref{sec:veering} that $K \subset S^3$ is an \emph{almost L-space knot} if it satisfies
\begin{equation} \label{eq:almost-l-space-knot}
\dim \hfhat(S^3_m(K)) = m+2
\end{equation}
for all $m \gg 0$.  Since \eqref{eq:almost-l-space-knot} is equivalent (via \eqref{eq:hfhat-odd-from-hfhat}) to \[\dim \hfhat_\godd(S^3_m(K)) = 1\] for $m \gg 0$, Theorem~\ref{thm:trace-large-surgery} says that for each $n\geq 0$, the $n$-trace $X_n(K)$ detects whether $K$ is an almost L-space knot.    In fact, more is true: we further showed in \cite[Theorem~1.11]{bs-characterizing} that the only almost L-space knots of genus $1$ are the left-handed trefoil, the figure eight, and $\mirror{5_2}$; and that almost L-space knots of genus $g\geq 2$ are fibered and strongly quasipositive.  The latter implies that $(\hat{r}_0(K),\hat\nu(K))=(2g+1,2g-1)$ by \cite[Corollary~3.11]{bs-characterizing}, and moreover that $\nu(K) = g$, so now case~\eqref{i:n-nu-large} of Theorem~\ref{thm:trace-determines-hfhat-restated} yields the following.

\begin{theorem} \label{thm:almost-l-space}
Let $K$ be an almost L-space knot of genus at least 2, and suppose for some knot $J \subset S^3$ and some $n\in\Z$ that $X_n(J) \cong X_n(K)$.  Then $J$ is an almost L-space knot, and $g(J)=g(K)$.
\end{theorem}

For genus-1 almost L-space knots, we note that $\partial X_n(K) \cong S^3_n(K)$ is already enough to detect the left-handed trefoil and the figure eight \cite{osz-characterizing}, and if $n \geq 0$ then $\partial X_n(K)$ also detects $\mirror{5_2}$ \cite{bs-characterizing}.  We do not know whether $X_n(K)$ can detect $\mirror{5_2}$ when $n$ is negative.

\section{Negative traces and positive torus knots} \label{sec:torus-surgeries}
The goal of this section is to prove Theorem~\ref{thm:negtracetorus}, which we restate here as follows:

\begin{theorem} \label{thm:negtracetorus-restated}
Let $T_{p,q}$ be a positive torus knot, and suppose for some knot $K \subset S^3$ and some integer $n \leq 0$ that $X_n(K) \cong X_n(T_{p,q})$.  Then $K = T_{p,q}$.
\end{theorem}

 For $n=0$ this follows from Theorem \ref{thm:main}, since positive torus knots are L-space knots. We  prove this theorem below, assuming some propositions which we will establish later in this section:

\begin{proof}[Proof of Theorem \ref{thm:negtracetorus-restated}]
Comparing the boundaries of the traces, we must have $S^3_n(K) \cong S^3_n(T_{p,q})$.  Thus if either $K$ or $T_{p,q}$ is unknotted then the other one must be as well, because all slopes are characterizing for the unknot \cite{kmos}.  From now on we assume that both knots are nontrivial, and we observe that since $T_{p,q}$ is an L-space knot, Theorem~\ref{thm:lspace-detection-main} says that $K$ must also be an L-space knot -- in particular, it is fibered \cite{ghiggini,ni-hfk} -- and that $g(K) = g(T_{p,q})$.

We know from geometrization for Haken manifolds that $K$ must be either a torus knot, a satellite knot, or a hyperbolic knot \cite[Corollary~2.5]{thurston-kleinian}.  In each case, the identification $S^3_n(K) \cong S^3_n(T_{p,q})$ will give us a contradiction:
\begin{itemize}
\item Proposition~\ref{prop:compare-torus-surgeries} says that if $K \neq T_{p,q}$ is a torus knot, then $g(K) \neq g(T_{p,q})$.
\item Proposition~\ref{prop:compare-satellite-surgeries} says that if $K$ is a satellite L-space knot, then $n \geq 2g(K)-1$.
\item Proposition~\ref{prop:compare-hyperbolic-surgeries} says that if $K$ is a hyperbolic L-space knot then $0 < n < 4g(K)$.
\end{itemize}
Since $K$ is indeed an L-space knot and $g(K) = g(T_{p,q})$, this means that the only way we can have $X_n(K) \cong X_n(T_{p,q})$ for any $n \leq 0$ is if $K = T_{p,q}$.
\end{proof}

The following subsections are dedicated to the proofs of Propositions~\ref{prop:compare-torus-surgeries}, \ref{prop:compare-satellite-surgeries}, and \ref{prop:compare-hyperbolic-surgeries}.

\subsection{Torus knots}

It is possible for two different torus knots to have a common surgery; in Appendix~\ref{sec:torus-pairs} we will prove Theorem~\ref{thm:torus-pairs}, providing a two-parameter family of examples where the resulting 3-manifold is always a lens space.  However, in these examples we are about to show that the torus knots in question cannot have the same Seifert genus, so Theorem~\ref{thm:lspace-detection-main} will tell us that the corresponding traces must be different.

\begin{proposition} \label{prop:compare-torus-surgeries}
Suppose for some $r \in \Q$ and some nontrivial torus knots $T_{a,b} \neq T_{c,d}$ that $S^3_r(T_{a,b}) \cong S^3_r(T_{c,d})$.  Then $r\neq 0$ and $g(T_{a,b}) \neq g(T_{c,d})$.
\end{proposition}

\begin{proof}
Let $A,B,C,D$ denote the absolute values of $a,b,c,d$ for convenience.  The case $r=0$ is ruled out by Lemma~\ref{lem:torus-knots-different-0-surgery}, so we will suppose that $r \neq 0$; then we can apply the surgery formula for the Casson--Walker invariant \cite{walker}, which says that
\[ \lambda(S^3_r(K)) = \lambda(S^3_r(U)) + \frac{1}{r}\frac{\Delta''_K(1)}{2}. \]
Since $S^3_r(T_{a,b}) \cong S^3_r(T_{c,d})$, it follows that
\[ \tfrac{1}{2}\Delta''_{T_{a,b}}(K) = \tfrac{1}{2}\Delta''_{T_{c,d}}(K). \]
These values are known explicitly: one can show as in, for example, \cite[Appendix~A]{mccoy-pretzel} that
\begin{equation} \label{eq:torus-casson}
\frac{\Delta_{T_{a,b}}''(1)}{2} = \frac{(a^2-1)(b^2-1)}{24} = g(T_{a,b}) \cdot \frac{(A+1)(B+1)}{12},
\end{equation}
and likewise for $T_{c,d}$.

If we assume that $g(T_{a,b}) = g(T_{c,d})$, then we get the system of equations
\begin{align}
(A-1)(B-1) &= (C-1)(D-1) \label{eq:torus-g-equal} \\
(A+1)(B+1) &= (C+1)(D+1) \label{eq:torus-g-plus}
\end{align}
where \eqref{eq:torus-g-equal} is equivalent to $2g(T_{a,b})=2g(T_{c,d})$, and \eqref{eq:torus-g-plus} comes from applying \eqref{eq:torus-casson} to $\Delta_{T_{a,b}}''(1) = \Delta_{T_{c,d}}''(1)$ and then dividing by \eqref{eq:torus-g-equal}.  Subtracting \eqref{eq:torus-g-equal} from \eqref{eq:torus-g-plus} gives us
\[ 2(A+B) = 2(C+D), \]
so $A+B=C+D$, and if we add this to \eqref{eq:torus-g-equal} then we get $AB=CD$ as well.  Now we have $\{A,B\} = \{C,D\}$, because these are the sets of roots of the quadratic polynomials
\[ x^2 - (A+B)x + AB = x^2 - (C+D)x + CD, \]
so it must be the case that $T_{c,d} = T_{-a,b}$.

We have now shown that $\{T_{a,b},T_{c,d}\} = \{T_{A,B},T_{-A,B}\}$, and thus
\[ S^3_r(T_{A,B}) \cong S^3_r(T_{-A,B}) \cong -S^3_{-r}(T_{A,B}). \]
The dimension of $\hfhat$ is preserved by orientation reversal, so it follows that
\begin{equation} \label{eq:hfhat-pm-r}
\dim \hfhat(S^3_r(T_{A,B})) = \dim \hfhat(S^3_{-r}(T_{A,B})).
\end{equation}
Since $T_{A,B}$ is an L-space knot of genus $g = g(T_{a,b})$, we have
\[ \hat{r}_0(T_{A,B}) = \hat\nu(T_{A,B}) = 2g-1, \]
see e.g.\ \cite[Remark~10.9]{bs-concordance}.  Writing $r = p/q$ for coprime integers $p,q$ with $p\neq 0$ and $q > 0$, we apply Proposition~\ref{prop:hfhat-surgeries} to \eqref{eq:hfhat-pm-r} to see that
\[ q(2g-1) + |p-q(2g-1)| = q(2g-1) + |{-p}-q(2g-1)|, \]
which simplifies to
\[ p-q(2g-1) = \pm(p+q(2g-1)). \]
But this is only possible if either $p=0$ or $q=0$, and in either case we have a contradiction.  Thus we must have had $g(T_{a,b}) \neq g(T_{c,d})$ after all.
\end{proof}

\subsection{Satellite knots}

In this subsection we show that if $K$ is a satellite L-space knot of the same genus as $T_{p,q}$, and if $S^3_r(K) \cong S^3_r(T_{p,q})$, then $r$ must be large enough so that their common $r$-surgery is an L-space.

\begin{lemma} \label{lem:prelim-satellite-surgeries}
Suppose for some nontrivial satellite knot $K=P(C)$ and slope $r\in\Q$ that the companion torus $\partial \nu(C)$ is compressible in $S^3_r(K)$.  Then $r\neq 0$, and if $S^3_r(K)$ is irreducible then 
\begin{itemize}
\item the pattern $P \subset S^1\times D^2$ has a nontrivial $S^1\times D^2$ surgery, and is therefore a 0-bridge braid or a 1-bridge braid, with winding number $w \geq 2$ in either case;
\item $r$-surgery on the pattern $P \subset S^1\times D^2$ is a solid torus;
\item and $S^3_r(K) \cong S^3_{r/w^2}(C)$.
\end{itemize}
\end{lemma}

\begin{proof}
Since the companion torus $\partial \nu(C)$ in $S^3 \setminus \nu(K)$ compresses in $S^3_r(K)$, the result $Z$ of $r$-surgery on the pattern $P \subset S^1\times D^2$ must have compressible boundary.  Thus Gabai \cite[Theorem~1.1]{gabai-solid-tori} showed that either
\begin{itemize}
\item $Z \cong S^1\times D^2$, and $P$ is either a $0$-bridge braid or a $1$-bridge braid; or
\item $Z \cong (S^1\times D^2) \# Y$, where $Y$ is closed with $|H_1(Y)|$ finite and nontrivial.
\end{itemize}
Scharlemann \cite{scharlemann-reducible} showed that in the latter case $P$ must be a cable of a $0$-bridge braid, and $Y$ is then a lens space.  Thus in either case $P$ has nonzero winding number $w \geq 1$.

Next, we determine how the $S^1\times D^2$ summand of $Z$ is glued to the exterior of the companion $C$ to form $S^3_r(K)$.  Writing $V = S^1\times D^2$ for the solid torus containing $P$, we know from \cite[Lemma~3.3]{gordon} (plus the fact that $w \neq 0$) that in $V_r(P) \cong Z$, the peripheral class
\[ a[\{\pt\} \times \partial D^2] + bw^2[S^1\times \{\pt\}] \]
is nullhomologous.  Since $\{\pt\}\times \partial D^2$ and $S^1\times \{\pt\}$ in $\partial V$ are identified with the meridian $\mu_C$ and longitude $\lambda_C$ of $C$, respectively, it follows that the $S^1\times D^2$ summand of $Z$ fills the exterior of $C$ along a curve of slope $\frac{a}{bw^2} = \frac{r}{w^2}$.  In other words, we have
\[ S^3_r(K) \cong S^3_{r/w^2}(C) \text{ or } S^3_{r/w^2}(C) \# Y \]
depending on whether $Z$ is $S^1\times D^2$ or $(S^1\times D^2) \# Y$.

We now claim that $Z \cong S^1\times D^2$.  Indeed, if we assume instead that $Z \cong (S^1\times D^2) \# Y$, then we have seen that $S^3_r(K) \cong S^3_{r/w^2}(C) \# Y$; but $C$ being nontrivial means that $S^3_{r/w^2}(C) \not\cong S^3$, and so $S^3_r(K)$ would then be reducible.  Assuming from now on that $S^3_r(K)$ is irreducible, which is automatically satisfied if $r=0$ \cite{gabai-foliations3}, we then have $Z \cong S^1\times D^2$ as claimed, and so $S^3_r(K)$ is a Dehn filling of $C$.  It follows that
\[ S^3_r(K) \cong S^3_{r/w^2}(C), \]
and moreover that $P$ is a $0$-bridge braid or a $1$-bridge braid in $S^1\times D^2$.  The latter fact tells us that $w \geq 2$, because if we had $w=1$ then $P$ would have to be isotopic to the core of $S^1\times D^2$ and we know that it is not.

Finally, in the case $r=0$ we have $S^3_0(K) \cong S^3_0(C)$, so $g(K)=g(C)$ \cite{gabai-foliations3}.  On the other hand, since $w\geq 2$ we know that
\[ g(K) = g(P(U)) + w\cdot g(C) \geq 2g(C), \]
hence $g(C) \geq 2g(C)$ and we have a contradiction.  We conclude that $r \neq 0$ after all.
\end{proof}

\begin{lemma} \label{lem:torus-surgery-essential}
Let $T_{p,q}$ be a torus knot, with $|p|,|q| \geq 2$ relatively prime, and fix a slope $r \in \Q$.
\begin{enumerate}
\item The Dehn surgery $S^3_r(T_{p,q})$ is reducible if and only if $r=pq$. \label{i:torus-surgery-reducible}
\item If $S^3_r(T_{p,q})$ has an incompressible torus, then $r=0$ and $T_{p,q} = T_{\pm 2,3}$. \label{i:torus-surgery-atoroidal}
\end{enumerate}
\end{lemma}

\begin{proof}
Moser \cite{moser} showed that $S^3_{pq}(T_{p,q})$ is a connected sum of two lens spaces, of orders $p$ and $q$; and that if $r \neq pq$ then $S^3_r(T_{p,q})$ is small Seifert fibered with base orbifold $S^2(p,q,\Delta(pq,r))$, hence irreducible.  This establishes \eqref{i:torus-surgery-reducible}.

For \eqref{i:torus-surgery-atoroidal}, we can assume that $r \neq pq$ since a sum of lens spaces is atoroidal.  Since $S^3_r(T_{p,q})$ is small Seifert fibered it has no vertical incompressible tori, so it must be covered by $T^3$ \cite[Corollary~10.4.13]{martelli} and therefore admit a flat metric.  Up to orientation, the only such $3$-manifold with cyclic first homology is $S^3_0(T_{2,3})$, by \cite[Theorem~12.3.1]{martelli} and \cite[Proposition~10.3.38]{martelli}.  But if $S^3_r(T_{p,q}) \cong \pm S^3_0(T_{2,3})$ then $r=0$ for homological reasons, and by \cite{moser} the two sides of this identification are both Seifert fibered with base orbifolds $S^2(|p|,|q|,|pq|) \cong S^2(2,3,6)$, so we must have $\{|p|,|q|\} = \{2,3\}$ and therefore $T_{p,q}$ is a trefoil.
\end{proof}

\begin{proposition} \label{prop:compare-satellite-surgeries}
Fix a positive torus knot $T_{p,q}$, where $p,q\geq 2$ are relatively prime positive integers, and suppose for some nontrivial satellite knot $K$ and rational $r \in \Q$ that $S^3_r(K) \cong S^3_r(T_{p,q})$.  If $K$ is an L-space knot, then $r \geq 2g(K)-1$.
\end{proposition}

\begin{proof}
We write $K = P(C)$, with pattern $P$ and companion $C$.  If $r=pq$ then $S^3_r(K) \cong S^3_r(T_{p,q})$ is a connected sum of two lens spaces, of orders $p$ and $q$ \cite{moser}, so $S^3_r(K)$ is an L-space and therefore $r \geq 2g(K)-1$.  Thus we can assume from now on that $r \neq pq$.  In particular, Lemma~\ref{lem:torus-surgery-essential} says that $S^3_r(K) \cong S^3_r(T_{p,q})$ is irreducible, and that it is atoroidal unless $S^3_r(K) \cong S^3_0(T_{2,3})$.  In the latter case we must have $r=0$, and then $K$ must be fibered of genus 1 \cite{gabai-foliations3}; but the only genus-1 fibered knots are the trefoils and the figure eight, and none of these are satellites, so in fact $S^3_r(K) \cong S^3_r(T_{p,q})$ must be atoroidal.

We have now shown that the companion torus must compress in $S^3_r(K)$, so we can apply Lemma~\ref{lem:prelim-satellite-surgeries} to see that $r$-surgery on $P$ is a solid torus -- in other words, that $P$ is a \emph{Berge--Gabai knot} -- and that
\[ S^3_r(K) \cong S^3_{r/w^2}(C), \]
where $w \geq 2$ is the winding number of $P$.  As a consequence of $P$ being a Berge--Gabai knot, Lemma~\ref{lem:prelim-satellite-surgeries} tells us via \cite{gabai-solid-tori} that the pattern $P$ is either a torus knot (i.e., a $0$-bridge braid) or a $1$-bridge braid; we will handle each of these cases separately, bearing in mind that we are not interested in arbitrary $1$-bridge braids but only those with solid torus surgeries.

First, if $P$ is a torus knot then $K$ is some $(m,n)$-cable of the companion $C$, where $\gcd(m,n)=1$ and $n = w \geq 2$.  Then since $K$ is an L-space knot, Hom \cite{hom-cabling} proved that $C$ is an L-space knot and that
\[ \frac{m}{n} \geq 2g(C) - 1; \]
in fact, the inequality must be strict since $n \nmid m$, so $\frac{m-1}{n} \geq 2g(C)-1$.  Moreover, we know from \cite[Corollary~7.3]{gordon} that $S^3_r(K)$ is only irreducible and atoroidal if $r=mn\pm\frac{1}{k}$ for some $k\in\Z$, in which case
\[ S^3_r(K) \cong S^3_{r/n^2}(C). \]
Now we have
\[ \frac{r}{n^2} = \frac{mn\pm\frac{1}{k}}{n^2} \geq \frac{m}{n} - \frac{1}{n^2} \geq \frac{m-1}{n} \geq 2g(C)-1, \]
so $S^3_r(K) \cong S^3_{r/n^2}(C)$ is an L-space, hence $r \geq 2g(K)-1$.

More generally, Hom--Lidman--Vafaee \cite[Theorem~1.3]{hom-lidman-vafaee} classified the L-space satellite knots whose patterns are Berge--Gabai knots.  Each Berge--Gabai knot $P$ is described by a triple of integers
\[ (w,b,t), \qquad 0 \leq b \leq w-2, \]
which are the winding number, bridge width, and twist number of $P$ respectively; they show that $K = P(C)$ is an L-space knot if and only if $C$ is an L-space knot and $b+tw \geq w^2(2g(C)-1)$.  Moreover, if the Berge--Gabai knot $P$ is specifically a 1-bridge braid, then \cite[Lemma~2.1]{hom-lidman-vafaee} says that $r$-surgery on $P$ is a solid torus only if
\[ r = tw+d, \qquad d \in \{b,b+1\}. \]
So in this case we have
\[ \frac{r}{w^2} \geq \frac{b+tw}{w^2} \geq 2g(C)-1 \]
and hence $S^3_r(K) \cong S^3_{r/w^2}(C)$ is an L-space.  This means that $r \geq 2g(K)-1$ once again, and we are done.
\end{proof}

\subsection{Hyperbolic knots}

In this subsection we determine when a fibered hyperbolic knot can have some surgery in common with a torus knot.

\begin{proposition} \label{prop:hyperbolic}
If $S^3_0(K) \cong S^3_0(T_{p,q})$ but $K$ is not isotopic to $T_{p,q}$, then $K$ is a fibered hyperbolic knot, and its monodromy is not veering.  In particular, neither $K$ nor its mirror is an L-space knot.
\end{proposition}

\begin{proof}
Since 0-surgery detects whether a knot is fibered \cite{gabai-foliations3}, and $T_{p,q}$ is fibered, we know that $K$ must be fibered as well.  Now if $K$ is not hyperbolic then it is either a torus knot or a satellite knot \cite[Corollary~2.5]{thurston-kleinian}.

If $K$ is a torus knot other than $T_{p,q}$, then Lemma~\ref{lem:torus-knots-different-0-surgery} says that $S^3_0(K) \not\cong S^3_0(T_{p,q})$.  Similarly, if $K$ is a satellite knot, then we cannot have $S^3_0(K) \cong S^3_0(T_{p,q})$.  Assuming otherwise, Lemma~\ref{lem:prelim-satellite-surgeries} says that the companion torus of $K=P(C)$ is incompressible in $S^3_0(K)$, so then $S^3_0(T_{p,q})$ must also be toroidal, whence $S^3_0(K) \cong S^3_0(T_{\pm2,3})$ by Lemma~\ref{lem:torus-surgery-essential}.  But then $K$ must be a genus-1 fibered knot, since it has the same $0$-surgery as another genus-1 fibered knot \cite{gabai-foliations3}, and we have a contradiction because the only genus-1 fibered knots are not satellites.

The only remaining possibility if $K\neq T_{p,q}$ is that it must be a fibered hyperbolic knot.  In this case, since $S^3_0(K) \cong S^3_0(T_{p,q})$ is small Seifert fibered, Corollary~\ref{cor:not-veering} tells us that the monodromy $h$ of $K$ cannot be veering, as claimed.  Moreover, if $K$ were an L-space knot then $h$ would have to be veering by Lemma~\ref{lem:ffpf-examples}, so $K$ cannot be an L-space knot.  Similarly, if the mirror $\mirror{K}$ were an L-space knot then its monodromy $h^{-1}$ would be veering, and then $h$ would be veering, so $\mirror{K}$ cannot be an L-space knot either.
\end{proof}

Proposition~\ref{prop:hyperbolic} has the following consequence, for which we recall that L-space knots have right-veering monodromy.

\begin{proposition} \label{prop:compare-hyperbolic-surgeries}
Let $T_{p,q}$ be a torus knot, and suppose there is some rational $r\in\Q$ and some fibered hyperbolic knot $K \subset S^3$ with right-veering monodromy such that $S^3_r(K) \cong S^3_r(T_{p,q})$.  Then $0 < r \leq 4g(K)$, and if $K$ is additionally an L-space knot then $r \neq 4g(K)$.
\end{proposition}

\begin{proof}
Write $g=g(K)$ for convenience.  We know that $S^3_r(T_{p,q})$ is always either reducible, a lens space, or small Seifert fibered \cite{moser}, so $S^3_r(K)$ is not hyperbolic.  Since $K$ is fibered and hyperbolic with right-veering monodromy, Ni \cite[Theorem~1.1]{ni-exceptional} proved that $0 \leq r \leq 4g$, and that $S^3_{4g}(K)$ is not a small Seifert fibered L-space.  Now Proposition~\ref{prop:hyperbolic} says that $K$ cannot have right-veering monodromy if $S^3_0(K) \cong S^3_0(T_{p,q})$, so we must have $r\neq 0$ as well.

In the case $r=4g$, we observe that $S^3_{4g}(K)$ is neither reducible nor a lens space.  Indeed, if we assume otherwise then its universal cover is not contractible, so it does not admit an essential lamination \cite[Theorem~6.1]{gabai-oertel}.  On the other hand, the monodromy of $K$ is pseudo-Anosov, and the suspension of its stable lamination has some degeneracy slope $d$, which satisfies either $|d| \leq 4g-2$ or $d=\frac{n}{0}$ for some integer $n\geq 1$ \cite[Theorem~8.8]{gabai-problems}.  Now by \cite[Theorem~5.3]{gabai-oertel} this suspension gives rise to an essential lamination on $S^3_{4g}(K)$ as long as $\Delta(d,4g) \geq 2$, so we must have $\Delta(d,4g) \leq 1$, hence $d = \frac{1}{0}$.  On the other hand, since $K$ has right-veering monodromy $h$, we know that the fractional Dehn twist coefficient $c(h) = \frac{1}{d}$ is strictly positive \cite[Proposition~3.1]{hkm-veering} and this is a contradiction.

Thus $S^3_{4g}(K) \cong S^3_{4g}(T_{p,q})$ implies that $S^3_{4g}(K)$ can only be small Seifert fibered.  If $K$ were an L-space knot, then all surgeries on $K$ of slope at least $2g-1$ would be L-spaces; but then $S^3_{4g}(K)$ would be a small Seifert fibered L-space, contradicting \cite{ni-exceptional} as mentioned above, so $K$ cannot be an L-space knot after all.
\end{proof}

Proposition~\ref{prop:compare-hyperbolic-surgeries} was the last remaining step in the proof of Theorem~\ref{thm:negtracetorus-restated}, so the latter proof is now complete. \hfill\qed

\section{Positive traces and $T_{2,2g+1}$ torus knots} \label{sec:t2n-positive}

As a special case of Theorem~\ref{thm:main}, we know that all torus knots are characterized by their 0-traces, but the proof makes essential use of the fact that $S^3_0(T_{p,q})$ is a fibered 3-manifold.  Similarly, Theorem~\ref{thm:negtracetorus} tells us that positive torus knots are characterized by any of their negative traces, but reveals very little about their positive traces.  However, in some cases we can say a bit more about these positive traces.  In this section we will prove Theorem~\ref{thm:positive-trace-t2n-main}, which we restate here:

\begin{theorem} \label{thm:positive-trace-t2n}
Suppose for some knot $K$ and  positive integers $n$ and $g$ that $X_n(K) \cong X_n(T_{2,2g+1}).$  Then either $K = T_{2,2g+1},$ or   $1 \leq n \leq 4g-1$ and $K$ is a hyperbolic knot such that \[\hfkhat(K) \cong \hfkhat(T_{2,2g+1})\] as bigraded vector spaces.
\end{theorem}

\subsection{Alexander polynomials of L-space knots}

We begin with a computation that follows from useful restrictions found by Krcatovich \cite{krcatovich} on the Alexander polynomials of L-space knots.

\begin{theorem}[{\cite[Theorem~1.6]{krcatovich}}] \label{thm:krcatovich}
If $K \subset S^3$ is an L-space knot of genus $g \geq 1$, then its Alexander polynomial has the form
\begin{equation} \label{eq:krcatovich-lspace}
\Delta_K(t) = (1-t^{-1})\sum_{i=0}^\infty t^{a_i},
\end{equation}
where $\{a_i\}$ is a strictly decreasing sequence with $a_0 = g$ and $a_i=-i$ for all $i \geq g$, and moreover
\[ a_i \leq g-2i \]
for $0 \leq i \leq g$.
\end{theorem}

It will be convenient for us to rewrite \eqref{eq:krcatovich-lspace} slightly, using 
\[ (1-t^{-1})\sum_{i=g}^\infty t^{a_i} = (1-t^{-1})\sum_{i=g}^\infty t^{-i} = t^{-g} \]
to say that
\begin{equation} \label{eq:krcatovich-finite}
\Delta_K(t) = (1-t^{-1}) \left(\sum_{i=0}^{g-1} t^{a_i}\right) + t^{-g}.
\end{equation}

\begin{lemma} \label{lem:other-alexander-restrictions}
Under the hypotheses of Theorem~\ref{thm:krcatovich}, we have $a_{g-1} = 2-g$, and if $g \geq 2$ then the remaining exponents
\[ a_1 > a_2 > \dots > a_{g-2} \]
contain exactly one element from each pair $\{g-2,3-g\},\ \{g-3,4-g\},\ \{g-4,5-g\},\ \dots,\ \{1,0\}$.
\end{lemma}

\begin{proof}
We use \eqref{eq:krcatovich-finite} to write
\begin{align*}
\Delta_K(t^{-1}) &= t^g + (1-t)\sum_{i=0}^{g-1} t^{-a_i} \\
&= \left( (1-t^{-1})(t^g + t^{g-1} + \dots + t^{1-g}) + t^{-g} \right) - (1-t^{-1})\sum_{i=0}^{g-1} t^{1-a_i} \\
&= (1-t^{-1})\left(\sum_{\substack{1-g \leq b \leq g \\ b \neq 1-a_i\ \forall i}} t^b\right) + t^{-g}.
\end{align*}
Since $\Delta_K(t) = \Delta_K(t^{-1})$, we conclude that the $g$-element sets
\[ \{a_0,\dots,a_{g-1}\} \quad\text{and}\quad \{1-a_0,\dots,1-a_{g-1}\} \]
are disjoint and thus partition the $2g$-element set $\{g,g-1,\dots,1-g\}$.  In other words, no two elements of $\{a_0,\dots,a_{1-g}\}$ can sum to $1$.

This last condition implies that $a_{g-1} \neq 1-a_0 = 1-g$, and we have $a_{g-1} > a_g = -g$, so in fact $a_{g-1} \geq 2-g$.  At the same time, Theorem~\ref{thm:krcatovich} says that $a_{g-1} \leq g-2(g-1) = 2-g$, so equality must hold.  Then the remaining $g-2$ exponents $a_1 > a_2 > \dots > a_{g-2}$ contain at most one element from each pair $\{g-2,3-g\}, \{g-3,4-g\}, \dots, \{1,0\}$, since otherwise two of them would sum to $1$, but there are exactly $g-2$ such pairs so each must contain one of the $a_i$ as claimed.
\end{proof}

We will eventually be interested in the Casson--Walker invariants \cite{walker} of Dehn surgeries on L-space knots, so the following lemma will allow us to compute them in terms of the coefficients $a_i$.

\begin{lemma} \label{lem:lspace-casson}
If $K \subset S^3$ is an L-space knot of genus $g \geq 1$, and we write $\Delta_K(t)$ in the form \eqref{eq:krcatovich-lspace}, then
\[ \frac{\Delta''_K(1)}{2} = \sum_{i=0}^{g-1} a_i + \frac{g(g-1)}{2}. \]
\end{lemma}

\begin{proof}
Differentiating \eqref{eq:krcatovich-finite} twice yields
\begin{multline*}
\Delta''_K(t) = (1-t^{-1})\left(\sum_{i=0}^{g-1} a_i(a_i-1)t^{a_i-2}\right) \\
+ 2t^{-2}\left(\sum_{i=0}^{g-1} a_it^{a_i-1}\right) - 2t^{-3}\left(\sum_{i=0}^{g-1} t^{a_i}\right) + g(g+1)t^{-g-2},
\end{multline*}
and then we set $t=1$ and divide by $2$ to get the desired expression.
\end{proof}

We can now prove that the Alexander polynomial of $T_{2,2g+1}$ is uniquely identified among all Alexander polynomials of genus-$g$ L-space knots by its second derivative at $t=1$.

\begin{proposition} \label{prop:casson-t2n}
Let $K \subset S^3$ be an L-space knot of genus $g \geq 1$.  Then its Alexander polynomial satisfies
\[ \frac{\Delta''_K(1)}{2} \leq \frac{g(g+1)}{2}, \]
with equality if and only if $\Delta_K(t) = \Delta_{T_{2,2g+1}}(t)$.
\end{proposition}

\begin{proof}
We combine Lemma~\ref{lem:lspace-casson} with Krcatovich's inequality $a_i \leq g - 2i$ for $0 \leq i \leq g$ from Theorem~\ref{thm:krcatovich} to write
\begin{align*}
\frac{\Delta''_K(1)}{2} &\leq \sum_{i=0}^{g-1} (g-2i) + \frac{g(g-1)}{2} \\
&= \left(g^2 - g(g-1)\right) + \frac{g(g-1)}{2} = \frac{g(g+1)}{2},
\end{align*}
where equality holds if and only if $a_i = g-2i$ for each $i=0,1,\dots,g-1$.  But this is the case precisely when
\begin{align*}
\Delta_K(t) &= (1-t^{-1})(t^g + t^{g-2} + t^{g-4} + \dots + t^{2-g}) + t^{-g} \\
&= t^g - t^{g-1} + t^{g-2} - t^{g-3} + t^{g-4} - \dots + t^{2-g} - t^{1-g} + t^{-g} \\
&= \Delta_{T_{2,2g+1}}(t),
\end{align*}
as claimed.
\end{proof}

\begin{remark}
There are other cases where knowing the value of $\frac{1}{2}\Delta''_K(1)$ for some L-space knot $K$ is enough to recover $\Delta_K(t)$, even though it does not work in general.  For example, using Lemmas~\ref{lem:other-alexander-restrictions} and \ref{lem:lspace-casson}, one can directly check that if $J$ and $K$ are L-space knots of the same genus $g \leq 5$ satisfying $\frac{1}{2}\Delta''_J(1) = \frac{1}{2}\Delta''_K(1)$, then $\Delta_J(t) = \Delta_K(t)$, since each possible choice of $a_1>\dots>a_{g-2}$ leads to a different value of $\sum_i a_i$.
\end{remark}

In the next subsection, we will classify the non-hyperbolic L-space knots whose Alexander polynomials achieve equality in Proposition~\ref{prop:casson-t2n}.  Such knots must be either torus knots or satellite knots \cite[Corollary~2.5]{thurston-kleinian}, so it will also be helpful to compute $\frac{1}{2}\Delta''_K(1)$ for satellite knots $K$ in terms of the pattern and companion knots.

\begin{lemma} \label{lem:casson-satellite}
Suppose that $K \subset S^3$ is a nontrivial satellite knot, and $K = P(C)$ where $P \subset S^1\times D^2$ is the pattern and $C \subset S^3$ the companion.  Then
\[ \Delta''_K(1) = \Delta''_{P(U)}(1) + w^2\Delta''_C(1). \]
\end{lemma}

\begin{proof}
Let $w \geq 0$ be the winding number of $K$.  We differentiate the relation
\[ \Delta_K(t) = \Delta_{P(U)}(t) \cdot \Delta_C(t^w) \]
twice to get
\begin{multline*}
\Delta''_K(t) = \Delta''_{P(U)}(t)\Delta_C(t^w) + 2\Delta'_{P(U)}(t) \cdot wt^{w-1}\Delta'_C(t^w) \\
+ \Delta_{P(U)}(t)\left( w(w-1)t^{w-2}\Delta'_C(t^w) + w^2t^{2w-2}\Delta''_C(t^w) \right).
\end{multline*}
The desired relation follows from setting $t=1$ and noting that $\Delta_{P(U)}(1) = \Delta_C(1) = 1$ while $\Delta'_{P(U)}(1) = \Delta'_C(1) = 0$.
\end{proof}

\subsection{Knots with the same knot Floer homology as $T_{2,2g+1}$}

We know that $\hfkhat(K)$ detects the trefoil $T_{2,3}$ \cite{ghiggini} and the cinquefoil $T_{2,5}$ \cite{frw-cinquefoil}, but not whether it detects other knots $T_{2,2g+1}$ with $g \geq 3$.  However, we can combine Proposition~\ref{prop:casson-t2n} with work of Baker--Motegi \cite{baker-motegi-twist} on L-space satellites to show that any other knot whose knot Floer homology agrees with that of $T_{2,2g+1}$ must be hyperbolic.

\begin{proposition} \label{prop:satellite-t2n}
Let $K$ be an L-space knot of genus $g \geq 1$, and suppose that $\Delta_K(t) = \Delta_{T_{2,2g+1}}(t)$.  Then $K$ is not a satellite knot.
\end{proposition}

\begin{proof}
Suppose that $K$ is a satellite knot, say $K=P(C)$ where $P \subset S^1\times D^2$ is the pattern and $C \subset S^3$ the companion.  Since $K$ is an L-space knot, the pattern $P$ is braided by \cite[Theorem~1.17]{baker-motegi-twist}, so since $P$ is a nontrivial satellite pattern it must have winding number $w \geq 2$.

Since $K$ is a satellite L-space knot, we appeal to \cite[Theorem~1.15]{hrw} to see that $P(U)$ and $C$ are both L-space knots as well.  Writing $h=g(P(U))$ and $k = g(C) \geq 1$, so that $g = h+wk$, we apply Proposition~\ref{prop:casson-t2n} and Lemma~\ref{lem:casson-satellite} to get
\begin{align*}
g(g+1) = \Delta''_K(1) &= \Delta''_{P(U)}(1) + w^2 \Delta''_C(1) \\
&\leq h(h+1) + w^2k(k+1) \\
&= (h^2 + (wk)^2) + (h+w^2k) \\
&= (h+wk)^2 + (h+wk) + (w^2k - 2hwk - wk) \\
&= g^2+g + wk(w-(2h+1))
\end{align*}
which can only be true if $w \geq 2h+1$.

If we had a strict inequality $w > 2h+1$, then the first few terms of $\Delta_K(t)$ would have the form
\begin{align*}
\Delta_K(t) = \Delta_{P(U)}(t) \cdot \Delta_C(t^w) &= (t^h - \dots + t^{-h}) (t^{wk} + O(t^{w(k-1)})) \\
&= (t^h - \dots + t^{-h})t^{wk} + O(t^{wk-w+h}) \\
&= t^{wk+h} - \dots + t^{wk-h} + O(t^{wk-w+h}).
\end{align*}
But $w > 2h+1$ is equivalent to $wk-w+h < wk-h-1$, so then the $t^{wk-h-1}$-coefficient of $\Delta_K(t)$ is zero.  Since $\Delta_K(t)$ has nonzero terms of strictly lower degree -- the lowest-degree term is $t^{-g} = t^{-wk-h}$ -- this contradicts our assumption that $\Delta_K(t)$ is equal to $\Delta_{T_{2,2g+1}}(t)$, all of whose $t^i$-coefficients for $-g \leq i \leq g$ are nonzero.  We must therefore have $w=2h+1$.  (Note that since $w \geq 2$, this means in particular that $h \geq 1$.)

Finally, Baker--Motegi \cite[Theorem~7.3]{baker-motegi-twist} proved that since $K$ is a satellite L-space knot, it satisfies
\[ w(2k-1) < \frac{2h-1 + w}{w-1}. \]
Using $w=2h+1$, this inequality simplifies to
\[ (2h+1)(2k-1) < \frac{4h}{2h} = 2. \]
Since the left side is at least $3\cdot 1 = 3$, we have a contradiction.
\end{proof}

\subsection{Traces that detect knot Floer homology}

We now combine the results of the previous sections to prove that the $n$-trace of any $T_{2,2g+1}$ recovers the knot Floer homology of any knot with the same $n$-trace.

\begin{proposition} \label{prop:lspace-trace-hyperbolic}
Fix $g \geq 1$, and suppose for some knot $K \subset S^3$ and integer $n \in \Z$ that $X_n(K) \cong X_n(T_{2,2g+1})$.  Then
\[ \hfkhat(K) \cong \hfkhat(T_{2,2g+1}) \]
as bigraded $\F$-vector spaces, and if $K$ is not isotopic to $T_{2,2g+1}$ then it is hyperbolic.
\end{proposition}

\begin{proof}
If $n\leq 0$ then $K=T_{2,2g+1}$ by Theorem~\ref{thm:negtracetorus}, so from now on we assume that $n \geq 1$.

Since $T_{2,2g+1}$ is an L-space knot, Theorem~\ref{thm:lspace-detection-main} says that $K$ must be an L-space knot of the same genus as $T_{2,2g+1}$, namely $g$.  Since the knot Floer homology of an L-space knot is completely determined by its Alexander polynomial \cite{osz-lens}, we will show that $\Delta_K(t) = \Delta_{T_{2,2g+1}}(t)$.

The identification of the $n$-traces restricts to a homeomorphism $S^3_n(K) \cong S^3_n(T_{2,2g+1})$ between their boundaries, which are rational homology spheres.  The surgery formula for the Casson--Walker invariant \cite{walker} says that
\[ \lambda(S^3_n(K)) = \lambda(S^3_n(U)) + \frac{1}{n}\frac{\Delta''_K(1)}{2} \]
and likewise for $\lambda(S^3_n(T_{2,2g+1}))$, so since these two invariants are equal we must have
\[ \frac{\Delta''_K(1)}{2} = \frac{\Delta''_{T_{2,2g+1}}(1)}{2}. \]
But Proposition~\ref{prop:casson-t2n} says that this is only possible if $\Delta_K(t) = \Delta_{T_{2,2g+1}}(t)$, as promised, and so the knot Floer homologies of $K$ and $T_{2,2g+1}$ coincide as well.

Supposing now that $K$ is not hyperbolic, it must be either a torus knot or a satellite knot \cite[Corollary~2.5]{thurston-kleinian}.  But if $K$ is an L-space knot with $\Delta_K(t) = \Delta_{T_{2,2g+1}}(t)$, then we know by Proposition~\ref{prop:satellite-t2n} that $K$ cannot be a satellite knot.  Thus it must be a torus knot, and in this case $\hfkhat(K) \cong \hfkhat(T_{2,2g+1})$ implies that $K = T_{2,2g+1}$ after all.
\end{proof}

This suffices to complete the proof of Theorem~\ref{thm:positive-trace-t2n}.

\begin{proof}[Proof of Theorem~\ref{thm:positive-trace-t2n}]
Suppose that $X_n(K) \cong X_n(T_{2,2g+1})$ but that $K$ is not isotopic to $T_{2,2g+1}$.  Then by Proposition~\ref{prop:lspace-trace-hyperbolic} we know that $K$ is a hyperbolic L-space knot of genus $g$, with $\hfkhat(K) \cong \hfkhat(T_{2,2g+1})$.  In particular, since $K$ is an L-space knot, it is fibered with right-veering monodromy.

Identifying the boundaries of the $n$-traces now gives us a homeomorphism
\[ S^3_n(K) \cong S^3_n(T_{2,2g+1}). \]
Since $K$ is a fibered, hyperbolic L-space knot, it has right-veering monodromy, and then Proposition~\ref{prop:compare-hyperbolic-surgeries} tells us that $0 < n < 4g$, as claimed.
\end{proof}

\appendix

\section{Non-characterizing slopes for torus knots} \label{sec:torus-pairs}

Here we describe a sequence of monic integer polynomials $a_k(n), b_k(n), c_k(n), d_k(n)$ and $p_k(n)$ that determine a 2-parameter family of pairs of torus knots with lens space surgeries in common.  By Proposition~\ref{prop:compare-torus-surgeries} the torus knots in each pair must have different genera, although we do not compute these genera directly.

\begin{theorem} \label{thm:torus-pairs}
For each integer $k \geq -1$ we define polynomials $a_k, b_k, c_k, d_k, p_k, q_k \in \Z[n]$ by setting 
\[ (a _{-1},b _{-1},c _{-1},d _{-1},p _{-1},q _{-1}) = (-1,\ 1,\ 1,\ 1,\ 0,\ n+1) \]
and then letting
\begin{align*}
(a _k,\ b _k,\ c _k,\ d _k) &= \left( d _{k-1},\ \frac{q _{k-1}-1}{d _{k-1}},\ b _{k-1},\ \frac{q _{k-1}+1}{b _{k-1}} \right), \\
(p _k,\ q _k) &= \left(q _{k-1},\ \frac{b _k^2 d _k^2 - 1}{q _{k-1}}\right)
\end{align*}
for all $k \geq 0$.  Then each of $a_k,b_k,c_k,d_k,p_k,q_k$ lies in $\Z[n]$, and for $k\geq 0$ they are monic of degrees $k$, $k+1$, $k$, $k+1$, $2k+1$, and $2k+3$ respectively.  Moreover, when $k\geq 1$ we have $a_k \neq c_k$ and $b_k \neq d_k$, and in these cases we have a homeomorphism
\[ S^3_{p_k(n)}(T_{a_k(n),b_k(n)}) \cong S^3_{p_k(n)}(T_{c_k(n),d_k(n)}) \]
as surgeries on nontrivial torus knots for all integers $n \geq 2$.
\end{theorem}

\begin{remark}
Note that even though $S^3_{p_k}(T_{a_k,b_k}) \cong S^3_{p_k}(T_{c_k,d_k})$ in the theorem, Proposition \ref{prop:compare-torus-surgeries} says that these torus knots have different genera. Then Theorem \ref{thm:lspace-detection-main} implies that the corresponding traces are distinct, $X_{p_k}(T_{a_k,b_k}) \not\cong X_{p_k}(T_{c_k,d_k})$.
\end{remark}

\begin{remark}
The restriction $n\geq 2$ in Theorem~\ref{thm:torus-pairs} is only needed to ensure that each $T_{a_k,b_k}$ and $T_{c_k,d_k}$ is a nontrivial torus knot, since when $n=1$ we can check by induction that
\[ (a_k,b_k,c_k,d_k,p_k,q_k) = (2k+1,\ 1,\ 1,\ 2k+3,\ 2k+2,\ 2k+4) \]
for all $k \geq -1$, and then $b_k(1) = c_k(1) = 1$ implies that $T_{a_k(1),b_k(1)}$ and $T_{c_k(1),d_k(1)}$ are both unknots.  On the other hand, letting $F_\ell$ be the $\ell$th Fibonacci number (with $F_0=0$, $F_1=F_2=1$, and so on), one can show that if $n=2$ then
\begin{align*}
(a_k,b_k) &= (F_{2k}+F_{2k+2},\ F_{2k+3}) &
p_k &= F_{4k+4} \\
(c_k,d_k) &= (F_{2k+1},\ F_{2k+2}+F_{2k+4}) &
q_k &= F_{4k+8}
\end{align*}
for all $k\geq -1$; in this case we have $a_k(2),\ b_k(2),\ c_k(2),\ d_k(2) \geq 2$ for all $k\geq 1$.
\end{remark}

\begin{proof}[Proof of Theorem~\ref{thm:torus-pairs}]
A priori each of $a_k,b_k,c_k,d_k,p_k,q_k$ is only a rational function, but by construction these satisfy
\begin{equation} \label{eq:lens-slope-p_k}
a_k b_k + 1 = p_k = c_k d_k - 1
\end{equation}
and
\begin{equation} \label{eq:define-q_k}
b_k^2d_k^2 = p_k q_k + 1
\end{equation}
for all $k \geq 0$.  We compute directly that
\[ (a _0,b _0,c _0,d _0,p _0,q _0) = (1,\ n,\ 1,\ n+2,\ n+1,\ n^3+3n^2+n-1), \]
verifying the proposition for $k=0$.

We claim by induction that for $k\geq 0$, all of $a_k, b_k, c_k, d_k, p_k, q_k$ belong to $\Z[n]$ and have the claimed degrees with leading coefficient $1$.  By inspection this is clearly true for $k=0$.  For larger values of $k$, we note that $a_{k+1} = d_k$, $c_{k+1} = b_k$, and $p_{k+1} = q_k$ satisfy the claim by inductive hypothesis.  Then \eqref{eq:lens-slope-p_k} tells us that $b_k$ and $d_k$ are relatively prime to $p_k$ as elements of $\Z[n]$, so by combining \eqref{eq:lens-slope-p_k} and \eqref{eq:define-q_k} to get
\begin{align*}
p_k(q_k-1) &= (b_k^2d_k^2-1) - (c_kd_k-1) = d_k(b_k^2d_k-c_k) \\
p_k(q_k+1) &= (b_k^2d_k^2-1) + (a_kb_k+1) = b_k(b_kd_k^2+a_k),
\end{align*}
we see that $d_k$ divides $q_k-1$ and $b_k$ divides $q_k+1$.  This implies that $b_{k+1} = \frac{q_k-1}{d_k}$ and $d_{k+1} = \frac{q_k+1}{b_k}$ are in $\Z[n]$ as well, and it follows by induction that they each have degree $k+2 = (k+1)+1$ and leading coefficient $1$.  This leaves only $q_{k+1}$, for which we note that
\[ b_{k+1}^2 d_{k+1}^2 = \left(\frac{q_k-1}{d_k}\right)^2 \left(\frac{q_k+1}{b_k}\right)^2 \equiv (b_k^2d_k^2)^{-1} \pmod{q_k}, \]
and by \eqref{eq:define-q_k} the right side is $1\pmod{q_k}$, so that $q_{k+1} = \frac{b_{k+1}^2d_{k+1}^2 - 1}{q_k}$ is a polynomial as well, which must then also have leading coefficient $1$ and degree $4(k+2)-(2k+3) = 2(k+1)+3$.  This completes the induction.

We now establish the claim that $a_k \neq c_k$ and $b_k \neq d_k$ for any $k\geq 1$ by noting that \eqref{eq:lens-slope-p_k} implies that $c_kd_k - a_kb_k = 2$, from which $\gcd(a_k,c_k)$ and $\gcd(b_k,d_k)$ both divide $2$.  Since these are polynomials of strictly positive degree with no non-constant common divisors, we cannot have $a_k = c_k$ or $b_k = d_k$.

To conclude, we apply \cite{moser} to see that for each $k\geq 1$ and $n \geq 2$, we have
\[ S^3_{p_k}(T_{a_k,b_k}) \cong L(p_k,b_k^2) \cong L(p_k,d_k^2) \cong S^3_{p_k}(T_{c_k,d_k}). \]
The relation \eqref{eq:lens-slope-p_k} guarantees that these surgeries are the specified lens spaces, and these lens spaces are identified in the middle because \eqref{eq:define-q_k} implies that $b_k^2d_k^2 \equiv 1 \pmod{p_k}$.
\end{proof}

We compute the first few values of $a_k$ et al.~in terms of $n$ as follows:
\begingroup
\allowdisplaybreaks
\begin{align*}
(a_1,b_1) &= (n+2,\ n^2+n-1), \\*
(c_1,d_1) &= (n,\ n^2+3n+1), \\*
p_1 &= n^3+3n^2+n-1 \\[1em]
(a_2,b_2) &= (n^2+3n+1,\ n^3+2n^2-n-1), \\*
(c_2,d_2) &= (n^2+n-1,\ n^3+4n^2+3n-1), \\*
p_2 &= n^5+5n^4+6n^3-2n^2-4n \\[1em]
(a_3,b_3) &= (n^3+4n^2+3n-1,\ n^4+3n^3-3n), \\*
(c_3,d_3) &= (n^3+2n^2-n-1,\ n^4+5n^3+6n^2-n-2), \\*
p_3 &= n^7 + 7n^6 + 15n^5 + 5n^4 - 15n^3 - 9n^2 + 3n + 1 \\[1em]
(a_4,b_4) &= (n^4+5n^3+6n^2-n-2,\ n^5 + 4n^4 + 2n^3 - 5n^2 - 2n + 1), \\*
(c_4,d_4) &= (n^4+3n^3-3n,\ n^5 + 6n^4 + 10n^3 + n^2 - 6n - 1), \\*
p_4 &= n^9 + 9n^8 + 28n^7 + 28n^6 - 21n^5 - 49n^4 - 6n^3 + 18n^2 + 3n - 1 \\[1em]
(a_5,b_5) &= (n^5 + 6n^4 + 10n^3 + n^2 - 6n - 1,\ n^6 + 5n^5 + 5n^4 - 6n^3 - 7n^2 + 2n + 1) \\*
(c_5,d_5) &= (n^5 + 4n^4 + 2n^3 - 5n^2 - 2n + 1,\ n^6 + 7n^5 + 15n^4 + 6n^3 - 11n^2 - 6n + 1) \\*
p_5 &= n^{11} + 11n^{10} + 45n^9 + 75n^8 + 6n^7 - 126n^6 - 98n^5 + 50n^4 + 60n^3 - 4n^2 - 8n .
\end{align*}
\endgroup
The case $k=1$ is \cite[Example~1.1]{ni-zhang-torus-1}, up to a slight change of variables; setting $n=2$ there yields the relation $S^3_{21}(T_{4,5}) \cong S^3_{21}(T_{2,11})$.

\bibliographystyle{myalpha}
\bibliography{References}

\end{document}